\documentclass[12pt,a4paper]{amsart}
\usepackage{amsfonts, amsthm, amsmath, amssymb}
\usepackage[top=1in,bottom=1in,left=1in,right=1in]{geometry}
\usepackage[pdftex]{hyperref}

\newcommand{\bC}{\mathbb{C}}

\newcommand{\bF}{\mathbb{F}}

\newcommand{\bM}{\mathbb{M}}

\newcommand{\bQ}{\mathbb{Q}}
\newcommand{\bR}{\mathbb{R}}

\newcommand{\bZ}{\mathbb{Z}}
\newcommand{\unit}{\mathbf{1}}

\newcommand{\fg}{\mathfrak{g}}

\newcommand{\fh}{\mathfrak{h}}
\newcommand{\fH}{\mathfrak{H}}
\newcommand{\fm}{\mathfrak{m}}
\newcommand{\fn}{\mathfrak{n}}

\newcommand{\simto}{\overset{\sim}{\to}}

\DeclareMathOperator{\ad}{ad}
\DeclareMathOperator{\Aut}{Aut}

\DeclareMathOperator{\Hom}{Hom}
\DeclareMathOperator{\id}{id}

\DeclareMathOperator{\rank}{rank}
\DeclareMathOperator{\Rep}{Rep}

\DeclareMathOperator{\Sym}{Sym}
\DeclareMathOperator{\Tr}{Tr}

\newcommand{\uRep}{\underline{\operatorname{Rep}}}

\newtheorem{thm}{Theorem}[section]
\newtheorem{prop}[thm]{Proposition}
\newtheorem{lem}[thm]{Lemma}
\newtheorem{cor}[thm]{Corollary}
\newtheorem{conj}[thm]{Conjecture}
\newtheorem*{conjj}{Conjecture}
\newtheorem*{thmm}{Theorem}

\theoremstyle{definition}
\newtheorem{defn}[thm]{Definition}

\theoremstyle{remark}
\newtheorem{rem}[thm]{Remark}

\begin{document}

\title{Monstrous Moonshine for integral group rings}
\author{Scott Carnahan, Satoru Urano}

\begin{abstract}
We propose a conjecture that is a substantial generalization of the genus zero assertions in both Monstrous Moonshine and Modular Moonshine.  Our conjecture essentially asserts that if we are given any homomorphism to the complex numbers from a representation ring of a group ring for a subgroup of the Monster, we obtain a Hauptmodul by applying this homomorphism to a self-dual integral form of the Moonshine module.  We reduce this conjecture to the genus-zero problem for ``quasi-replicable'' functions, by applying Borcherds's integral form of the Goddard-Thorn no-ghost theorem together with some analysis of the Laplacian on an integral form of the Monster Lie algebra.  We prove our conjecture for cyclic subgroups of the Monster generated by elements in class 4A, and we explicitly determine the multiplicities for a decomposition of the integral Moonshine Module into indecomposable modules of the integral group rings for these groups.
\end{abstract}

\maketitle

\tableofcontents

\section{Introduction}

In this paper, we propose a new monstrous Moonshine conjecture, and prove some partial results.  Our conjecture captures the Hauptmodul assertions in both Conway and Norton's Monstrous Moonshine conjecture and Ryba's Modular Moonshine conjecture, replacing traces or Brauer characters of automorphisms with maps from representation rings.  Our generalization allows for new phenomena from non-semisimple representations over rings of integers, and we consider arbitrary subgroups of the Monster instead of just cyclic groups.  

\subsection{Short history}

Let us briefly recall the history of this question.  In 1978, McKay observed a numerical relationship between the coefficients in the $q$-expansion of the modular function $j(\tau) = q^{-1} + 744 + 196884q + \cdots$ (where $q = e^{2\pi i \tau}$ and $\tau$ lies in the complex upper half-plane $\fH$), and the dimensions $1, 196883, \ldots$ of the smallest irreducible representations of the Monster simple group $\bM$.  Thompson then produced more numerical evidence for a relationship, and suggested the existence of a natural graded representation of $\bM$ whose graded dimensions yield the $q$-expansion of $j-744$ \cite{T79b}.  Furthermore, he suggested that the graded traces of elements of $\bM$ on such a representation may be interesting.  Conway and Norton \cite{CN79} and Thompson \cite{T79a} developed this idea into the following concrete conjecture :

\begin{conjj} (Monstrous Moonshine)
There is a natural infinite dimensional graded representation $V = \bigoplus_{n=0}^\infty V_n$ of $\bM$, such that for any $g \in \bM$, the series $\sum \Tr(g|V_n)q^{n-1}$ is the $q$-expansion of a modular function $T_g(\tau)$ on the complex upper half-plane $\fH$ that satisfies the following properties:
\begin{enumerate}
\item $T_g(\tau)$ is invariant under a discrete subgroup $\Gamma_g$ of $SL_2(\bR)$ that contains $\Gamma_0(12|g|)$.
\item $T_g(\tau)$ generates the field of meromorphic functions on the smooth compactification of the quotient Riemann surface $\fH/\Gamma_g$, i.e., $T_g(\tau)$ is a Hauptmodul for $\Gamma_g$.  In particular, $\fH/\Gamma_g$ is a genus zero complex curve with finitely many punctures.
\end{enumerate}
\end{conjj}

In fact, after undertaking a massive computation, Conway and Norton assigned candidate modular functions to each of the 194 conjugacy classes in $\bM$.  In 1984, Frenkel, Lepowsky, and Meurman \cite{FLM84} announced a candidate representation, constructed in \cite{FLM85}, and they refined their construction in \cite{FLM88} to show that their representation, which they named $V^\natural$, is endowed with a vertex operator algebra structure whose automorphism group is precisely $\bM$.  Furthermore, they showed that for elements centralizing a 2B involution, the graded traces were equal to the candidate modular functions.  In 1992, Borcherds showed that the Conway-Norton conjecture holds for $V^\natural$, using (among other tools) a string-theoretic quantization functor \cite{B92}.

Mathematicians were asking about Moonshine phenomona for groups other than the Monster from near the beginning of the history of the subject: Conway and Norton describe some preliminary hints with subgroups of $\bM$ in \cite{CN79}, and Queen \cite{Q81} did substantial computations showing that some simple combinations of representations of sporadic groups produce the first terms in many Hauptmoduln.  These new phenomena now have two explanations:
\begin{enumerate}
\item Norton's 1987 Generalized Moonshine conjecture \cite{N87} asserts the existence of projective representations $V(g)$ of centralizers of elements in $\bM$ such that traces yield Hauptmoduln, and many of these centralizers often have sporadic subquotients.  This was proved in \cite{C12} following substantial advances about twisted $V^\natural$-modules in \cite{DLM97} and their fusion in \cite{vEMS}, together with a proof in the Baby Monster case \cite{H03}.
\item Ryba's 1994 Modular Moonshine conjecture \cite{R96} asserts the existence of vertex algebras over the finite field $\bF_p$ with actions of centralizers of order $p$ elements in $\bM$, such that the Brauer characters of $p$-regular elements are Hauptmoduln.  This was proved for odd primes $p$ in \cite{BR96}, \cite{B98}, \cite{B99} and for the prime 2 in \cite{BR96} under an assumption that was resolved in \cite{C17}.
\end{enumerate}
We note that $\bM$ plays a central role in both explanations - that is, while looking for phenomena away from the Monster, Conway, Norton, and Queen did calculations whose results suggested the Monster was still fundamentally connected but in a more general context.

\subsection{Motivation for this paper}

The work in this paper arose from our investigations into generalizations of the Modular Moonshine conjecture.  This conjecture originally asserted the existence of vertex algebras over $\bF_p$ with well-behaved symmetries, but was reinterpreted in \cite{BR96} as a claim about Tate cohomology of prime order automorphisms of a self-dual integral form of $V^\natural$.  Naturally, our first steps involved studying Tate cohomology of composite order automorphisms, and we were guided by a bold conjecture proposed by Borcherds that would unify Generalized Moonshine and Modular Moonshine.  This conjecture asserted the existence of projective representations of centralizers $C_\bM(g)$ over cyclotomic integer rings $\bZ[e^{2\pi i/n}]$, that produced both the $g$-twisted $V^\natural$-modules $V(g)$ when base changed to $\bC$ and finite characteristic vertex algebras ${}^gV$ when base changed to a finite ring.  Recently, the second author of this paper showed that Borcherds's conjecture is too ambitious as stated \cite{U20}.  In particular, there is a Fricke element of order 8 with nonzero degree 1 Tate cohomology, and this contradicts one of the immediate consequences of the conjecture, namely that Fricke elements have vanishing Tate cohomology in odd degree.  More generally, some analysis of Tate cohomology of elements of order 6 strongly suggests the non-existence of modules that behave properly in different residue characteristics.

Further investigation revealed that many of the nice properties satisfied by Tate cohomology of cyclic groups of prime order fail rather badly when considering composite order.  For example, we see in \cite{B98} Corollary 2.2 that there is a superspace isomorphism $\hat{H}^*(A \otimes B) \cong \hat{H}^*(A) \otimes \hat{H}^*(B)$ for $\bZ_p[\bZ/p\bZ]$-modules $A$ and $B$, but it is easy to show that this fails for representations of cyclic groups of composite order.  In particular, functions like ``total dimension of Tate cohomology'' fail to be homomorphisms on the representation ring.  We eventually concluded that the important part of Borcherds's investigation in \cite{B98} lies not in the Tate cohomology, but in the ring homomorphisms from representation rings.  It is after all the homomorphism property that allows Borcherds to employ powerful methods like his integral no-ghost theorem and Hodge theoretic analysis of the Monster Lie algebra.

We therefore consider arbitrary subrings $R$ of $\bC$, and arbitrary subgroups $G$ of $\bM$, and consider the power series induced by applying homomorphisms from representation rings to the complex numbers.  The ``representation ring'' $\Rep_R(G)$ is the group completion of the semiring of stable isomorphism classes of $R[G]$-modules that are $R$-torsion-free of finite rank, where addition is given by direct sum, and multiplication is given by tensor product.  When $X$ is a set of $R$-torsion-free $R[G]$-modules, we write $\Rep^X_R(G)$ for the subring generated by the elements of $X$ and their Adams operations.  We will focus on the case where $G$ is a subgroup of $\bM$, $R$ is a subring of $\bC$, and $V^\natural_R = V^\natural_\bZ \otimes R$ is the self-dual $R$-form for the Monster vertex operator algebra constructed in \cite{C17}.  In this case, we have a subring $\Rep^\natural_R(G) \subseteq \Rep_R(G)$, defined as the smallest subring containing the indecomposable $R[G]$-module direct summands of $V^\natural_{n,R}, (n \geq 0)$ and closed under Adams operations.

Returning to Borcherds's work, when $G$ is a cyclic subgroup of $\bM$ of prime order $p$ and $R \cong \bZ_p$ (or an unramified extension), one has exactly three homomorphisms $\phi: \Rep_R(G) \to \bC$, namely dimension (i.e., trace of identity), the trace of a generator of $G$, and the total dimension of Tate cohomology.  Borcherds found that applying any of these maps to the graded pieces of $V^\natural_R$ yields a Hauptmodul.  We therefore feel it is natural to propose the following conjecture:

\begin{conj} \label{conj:main} (Monstrous Moonshine for group rings)
Let $R$ be a subring of $\bC$, and let $G$ be a subgroup of $\bM$.  Then for any ring homomorphism $\phi: \Rep^\natural_R(G) \to \bC$, the ``generalized McKay-Thompson series''
\[ T_\phi(\tau) := \sum_{n \geq 0} \phi(V^\natural_{n,R}) q^{n-1} \]
is the $q$-expansion of a finite level Hauptmodul.  That is, the holomorphic function $T_\phi(\tau)$ on the complex upper half plane $\fH$ is invariant under a discrete subgroup $\Gamma_\phi < SL_2(\bR)$ containing $\Gamma_0(N)$ for some $N$, such that the quotient $\fH/\Gamma_\phi$ is genus zero with finitely many punctures, and the field of meromorphic functions on the smooth compactification is generated by $T_\phi$.  Here, $V^\natural_{n,R}$ is the $n$-eigenspace for the action of $L_0$ on $V^\natural_\bZ \otimes R$.
\end{conj}

Attacking this conjecture presents several challenges beyond what we see in monstrous Moonshine and Modular Moonshine.  The biggest problem is that modules over the group ring $R[G]$ can be very complicated.  For general $R$ and $G$, the group ring has wild representation type, meaning the problem of classifying modules is at least as hard as classifying all finite dimensional modules over all finitely generated associative algebras over a field (see e.g., \cite{KL01} for a discussion of wildness for algebras over rings).  We are somewhat saved by our restriction to $R$-torsion free modules, as the Jordan-Zassenhaus theorem (in the number ring case) gives us finitely many isomorphism types in a fixed rank.  Furthermore, by applying homomorphisms to $\bC$, we can pass from studying isomorphism classes of modules to a coarser equivalence that collapses ideal classes.  Finally, by restricting to $\Rep^\natural_R(G)$, i.e., classes that appear in $V^\natural_R$, we reduce the complexity of the problem even further.  However, we still don't know if the smallest cases $\bZ[\bZ/2\bZ \times \bZ/2\bZ]$ and $\bZ[\bZ/8\bZ]$ of rings with infinitely many indecomposable $\bZ$-free modules restrict to finite collections.

A second problem, as we will see in the discussion of results, is that we don't have a way to control our functions as well as Borcherds did in \cite{B92} and \cite{B98}.  The notion of ``quasi-replicability'' that we obtain is new and not well-understood.

\subsection{Main results}

Our first important theorem is an enhancement of Borcherds's integral no-ghost theorem \cite{B99}:

\begin{thmm} \textbf{\ref{thm:no-ghost}:}
Let $R$ be a subring of $\bC$, and let $V$ be a unitarizable $U^+(vir)_R$-module of half central charge 12 equipped with an action of a group $G$ that commutes with the Virasoro action, and with a Virasoro-invariant $G$-invariant bilinear form, such that $V$ is self-dual.  Let $\beta$ be an element of $I\!I_{1,1} \otimes R$ such that $(\beta,\gamma) \in R^\times$ for some $\gamma \in I\!I_{1,1}\otimes R$, let $\pi^{1,1}_\beta$ be the Heisenberg module attached to $\beta$, and let $H = V \otimes \pi^{1,1}_\beta$.  Let $P^1 = \{ v \in H | L_0 v = v, L_i v = 0\, \forall i > 0 \}$, and let $N^1$ be the radical of the inner product on $P^1$.  Then, $V^{1-(\beta,\beta)/2} \cong P^1/N^1$ as $R[G]$-modules with $G$-invariant bilinear form.
\end{thmm}

Applying this result to $V^\natural_\bZ$, we obtain a Lie algebra over $\bZ$ whose primitive root spaces are identified with homogeneous pieces of the vertex algebra.

\begin{thmm} \textbf{\ref{thm:properties-of-monster-Lie-algebra}:}
There exists an integral form $\fm_\bZ$ of the Monster Lie algebra, satisfying the following properties:
\begin{enumerate}
\item $\fm_{\bZ}$ is a $\bZ \times \bZ$-graded Borcherds-Kac-Moody Lie algebra over $\bZ$, with an invariant bilinear form that identifies the degree $(0,0)$ subspace with the even unimodular lattice $I\!I_{1,1}$.
\item $\fm_{\bZ}$ has a faithful Monster action by homogeneous inner-product-preserving Lie algebra automorphisms, such that the $\bZ[\bM]$-module structure of the degree $(m,n)$ root space depends only on the value of $mn$, when $(m,n)$ is primitive (because it is isomorphic to $V^\natural_{1+mn,\bZ}$).
\item The positive subalgebra $\fn_\bZ = \bigoplus_{m \geq 1, n \geq -1} \fm_{m,n,\bZ}$ has $\bZ_{\geq 0} \times \bZ_{\geq 0} \times \bZ$-graded homology $H^*(\fn_\bZ,\bZ) = \bigoplus_{i \geq 0, m \geq 0, n \in \bZ} H^i(\fn_\bZ,\bZ)_{m,n}$.  Furthermore, for any subgroup $G$ of $\bM$, and any $i \geq 0$, and all degrees $(m,n)$ such that $(m-1)n$ is coprime to $|G|$, the representation $H^i(\fn_\bZ,\bZ)_{m,n}$ of $G$ lies in the torsion ideal of the representation ring $\Rep_\bZ(G)$.  In particular, applying any ring homomorphism $\phi: \Rep^\natural_\bZ(G) \to \bC$ to $H^i(\fn_\bZ,\bZ)_{m,n}$ yields zero.
\end{enumerate}
\end{thmm}

From this homological vanishing theorem we obtain a complicated collection of relations on the coefficients of the power series $T_\phi$.  We call the power series satisfying these relations ``quasi-replicable'' (see Definition \ref{defn:quasi-replicable}), and our main theorem is the following:

\begin{thmm} \textbf{\ref{thm:main}:}
Let $R$ be a subring of $\bC$, and let $G$ be a subgroup of $\bM$.  Then, for any ring homomorphism $\phi: \Rep_R^\natural(G) \to \bC$, the power series $T_\phi$ is quasi-replicable of exponent $|G|$.  
\end{thmm}

Computational evidence suggests the quasi-replicability relations completely determine the coefficients beyond a finite set, and we conjecture that quasi-replicable power series are either Hauptmoduln or finite Laurent polynomials in the power series variable $q$.  Quasi-replicability is a weaker condition than Conway and Norton's notion of replicability, and in order to understand it, it seems we need more than simple modifications of the strategies for proving modularity of replicable functions (e.g., in \cite{C08} Corollary 5.4, or \cite{CG97} Theorem 1.1 for completely replicable functions).

As we mentioned, our conjecture generalizes the Hauptmodul claims in both Monstrous Moonshine and Modular Moonshine.  Monstrous Moonshine concerns the case $R = \bC$, because the ring homomorphisms $\Rep_\bC(G) \to \bC$ are precisely the traces of elements of $G$.  Modular Moonshine concerns the case $R \cong \bZ_p$ and $G$ is generated by an element $g$ of prime order $p$ and a commuting element $h$ of order $k$ coprime to $p$ - more precisely, the ring homomorphisms ``trace of $gh$'' and ``Brauer character of $h$ on total Tate cohomology of $g$'' coincide when $g$ is Fricke.  We also have an explicit result that goes beyond these cases:

\begin{thmm} \ref{thm:4A}
If $R$ is any subring of $\bC$, and $G$ is cyclic subgroup of $\bM$ generated by an element in conjugacy class 4A, then $T_\phi$ is one of the Hauptmoduln $T_{1A} = j-744$, $T_{2B} = \frac{\eta(\tau)^{24}}{\eta(2\tau)^{24}} + 24$, or $T_{4A} = \frac{\eta(2\tau)^{48}}{\eta(\tau)^{24}\eta(4\tau)^{24}}-24$.  Furthermore, the homogeneous pieces of $V^\natural_\bZ$ are given as a sum of indecomposable modules for $G$ with multiplicities explicitly given by linear combinations of these Hauptmoduln.
\end{thmm}

In some cases we obtain incomplete results:

\begin{thmm} \ref{thm:6A}
If $R$ is any subring of $\bC$, and $G$ is a cyclic subgroup of $\bM$ generated by an element in conjugacy class 6A then $T_\phi$ is either one of the Hauptmoduln $T_{1A}$, $T_{2A}$, $T_{3A}$, $T_{6A}$, or a quasi-replicable power series whose coefficients are bounded below by those of $T_{6A}$ and above by $T_{3A}$.
\end{thmm}

We expect that this extra power series is in fact equal to $T_{6A}$, and that more generally, the analogous result holds for class $pq$A where $p$ and $q$ are distinct primes.

The situation with our main theorem is similar to the result we would get from Borcherds's proof of the Monstrous Moonshine conjecture if we stopped just before the last step, i.e., if we didn't know anything about the character table of $\bM$.  Without the explicit comparison between the head characters conjectured in \cite{CN79} and the first few graded pieces of $V^\natural$ in section 9 of \cite{B92}, we would only have the result that the McKay-Thompson series are completely replicable functions of finite order.  The main theorem of \cite{CG97} then implies the McKay-Thompson series are either Hauptmoduln or ``modular fictions'', that is, functions of the form $q^{-1} + \epsilon q$ for $\epsilon$ either 0 or a 24th root of unity.  One difference in the complex case is that we can eliminate the modular fictions by appealing to the modularity results of \cite{DLM97}, which imply the expansions of trace functions at other cusps must come from characters of twisted modules.  We may therefore hope that the work of Dong-Li-Mason on twisted modules admits a generalization that could yield a way to eliminate non-modular functions without explicit computation.

Our computational experiments suggest several other phenomena involving the $\bZ[\bM]$-module structure of $V^\natural_\bZ$, and we describe these in the open problems section at the end of this paper.

\section{Preparation}

\subsection{Representation rings}

We begin with the results we need in the theory of integral representations.  All we need from this theory can be found in \cite{R70}, which is a nice overview of the state of the art as of 1970.

\begin{defn}
Let $R$ be an integral domain, and $G$ a finite group.  We define the \textbf{representation category} $\uRep_R(G)$ to be the monoidal category of $R[G]$-modules that are $R$-torsion-free of finite rank, with tensor product as monoidal structure.  We define the \textbf{representation ring} of $R[G]$ (also known as the \textbf{Green ring}), written $\Rep_R(G)$, as the group completion of the semiring of isomorphism classes in $\uRep_R(G)$, together with the operations of direct sum as addition and tensor product as multiplication.  We define the subring $\Rep^\natural_R(G) \subseteq \Rep_R(G)$ as the smallest subring containing the indecomposable $R[G]$-module direct summands of $V^\natural_{n,R}, (n \geq 0)$ and closed under Adams operations (see Definition \ref{defn:adams-operations}).
\end{defn}

For general $R$ and large $G$, it is difficult to get any grasp on the structure of $\Rep_R(G)$.  For example, $\bZ[G]$ has infinitely many isomorphism types of indecomposable representations whenever $G$ has a non-cyclic Sylow subgroup \cite{HR62}, or the order of $G$ is a multiple of a nontrivial cube \cite{HR63}.  Furthermore, $R[G]$-modules do not necessarily satisfy the Krull-Schmidt property, i.e., we do not necessarily have unique decomposition into a sum of indecomposable modules \cite{R61}.  

However, we can avoid some of the problems when we consider ring homomorphisms $\Rep^R(G) \to \bC$, because torsion classes are necessarily sent to zero.  Indeed, the next result implies that the torsion is precisely what we lose by inverting all primes outside the support of $|G|$.

\begin{thm} \label{thm:can-invert-good-primes} (Main theorem of \cite{R67}, see also section 16 of \cite{R70})
Let $R$ be a Dedekind domain whose quotient field $F$ is an algebraic number field, let $G$ be a finite group, and let $R' = \bigcap_{P \supset |G|R} R_P$ be the semi-local ring given by inverting all elements coprime to $|G|$.  Then, the additive map $\Rep_R(G) \to \Rep_{R'}(G)$ given by base change has kernel equal to the subgroup of torsion elements.  Furthermore, this subgroup is an ideal in $\Rep_R(G)$, and equal to the finite set whose elements are $[R[G]] - [M]$ as $M$ ranges over all $R$-torsion-free $R[G]$-modules satisfying $M_P \cong R_P[G]$ for all primes $P$ of $R$.
\end{thm}


\begin{cor} \label{cor:homs-factor-through}
Let $R$ be a subring of $\bC$, let $G$ be a finite group, and let $R' = \bigcap_{P \supset |G|R} R_P$.  Then, any ring homomorphism from $\Rep^\natural_R(G)$ to $\bC$ factors through a homomorphism to $\Rep^\natural_{R'}(G)$ given by base change on objects.
\end{cor}
\begin{proof}
Both $\Rep^\natural_R(G)$ and $\Rep^\natural_{R'}(G)$ are generated by objects defined over subrings of number fields, so it suffices to consider the case $R$ is a Dedekind domain whose quotient field $F$ is an algebraic number field.  Any homomorphism from $\Rep^\natural_R(G)$ to $\bC$ extends non-uniquely to a group homomorphism from $\Rep_R(G)$, and since torsion lies in the kernel, Theorem \ref{thm:can-invert-good-primes} yields a homomorphism $\Rep_{R'}(G) \to \bC$, that, when restricted to $\Rep^\natural_{R'}(G)$ is what we want.
\end{proof}

\begin{cor} \label{cor:cancellation-for-order-coprime-to-laplacian}
Let $R$ be a subring of $\bC$, and let $G$ be a finite group.  Suppose we are given a collection $\{ V_i \}_{i \in \bZ}$ of objects in $\uRep_R(G)$ with only finitely many nonzero, together with $R[G]$-module maps $d_i: V_i \to V_{i+1}$ and $\delta_i: V_i \to V_{i-1}$ satisfying $d_{i+1}d_i = \delta_{i-1}\delta_i = 0$ and $\delta_{i+1}d_i + d_{i-1}\delta_i = k \id_{V_i}$ for some $k \in R$ that is invertible in $R' = \bigcap_{P \supset |G|R} R_P$ (i.e., coprime to $|G|$), and all $i \in \bZ$.  Then, the homology of the complex $(V_i,d_i)$ lies in the torsion ideal of $\Rep_R(G)$, and after base change to $R'$, the homology vanishes.  Furthermore, for any ring homomorphism $\phi$ from a subring of $\Rep_R(G)$ containing all $V_i$ to $\bC$, the alternating sum $\sum_i (-1)^i \phi(V_i)$ vanishes.
\end{cor}
\begin{proof}
Base change to $R'$ makes $k$ is invertible in $R$, and then the vanishing of homology and the alternating sum follow from the argument in \cite{B98} Lemma 2.9: If $k$ is invertible, then $A_i$ is the direct sum of $d\delta A_i$ and $\delta d A_i$ (specifically, any $x$ is uniquely written as $d\delta k^{-1}x + \delta d k^{-1}x$), and $d$ is an isomorphism from $\delta d A_i$ to $d \delta A_{i+1}$.  Thus the sequence splits as a direct sum of isomorphisms.  We conclude that the homology of $(V_i,d_i)$ lies in the kernel of the base change homomorphism, which is the torsion ideal by Theorem \ref{thm:can-invert-good-primes}.  Then, by Corollary \ref{cor:homs-factor-through}, $\phi$ factors through the corresponding subring of $\Rep_{R'}(G)$.\end{proof}

\begin{rem}
Any set of distinct homomorphisms $\Rep_R(G) \to \bC$ is linearly independent - this is shown in Lemma 6.5 of \cite{BP84} in the modular setting, and the proof works here without change.  This together with orthogonality of characters implies that when $|G|$ is invertible in $R$, all homomorphisms from $\Rep_R(G)$ to $\bC$ are given by the traces of elements.
\end{rem}

\subsection{Conformal vertex algebras}

\begin{defn}
Let $R$ be a commutative ring.  A \textbf{vertex algebra} over $R$ is an $R$-module $V$ equipped with a distinguished vector $\unit$ and an $R$-linear multiplication map $V \otimes_R V \to V((z))$, written $a \otimes b \mapsto Y(a,z)b = \sum_{n \in \bZ} a_n b z^{-n-1}$, satisfying the following conditions:
\begin{enumerate}
\item For all $a \in V$, $Y(a,z)\unit \in a + zV[[z]]$.
\item The Jacobi identity: For any $r,s,t \in \bZ$, and any $u,v,w \in V$, 
\[ \sum_{i \geq 0} \binom{r}{i} (u_{t+i} v)_{r+s-i} w = \sum_{i \geq 0} (-1)^i \binom{t}{i} (u_{r+t-i}(v_{s+i}w) - (-1)^t v_{s+t-i}(u_{r+i} w)) \] 
\end{enumerate}
\end{defn}

\begin{defn}
Let $V$ be a vertex algebra over a commutative ring $R$.  A \textbf{conformal vector} of ``half central charge'' $\hat{c} \in R$ is an element $\omega \in V$ satisfying the following properties:
\begin{enumerate}
\item $\omega_0 v = v_{-2}\unit$ for all $v \in V$.
\item $\omega_1 \omega = 2\omega$
\item $\omega_3 \omega = \hat{c} \unit$
\item $\omega_i \omega = 0$ for $i = 2$ or $i > 3$.
\item $L_0 = \omega_1$ acts semisimply with integer eigenvalues.
\end{enumerate}
We typically write $Y(\omega,z) = \sum_{n \in \bZ} L_n z^{-n-1}$, i.e., $L_n = \omega_{n+1}$.  A conformal vertex algebra is a \textbf{vertex operator algebra} if the $L_0$-eigenspaces are projective $R$-modules of finite rank.
\end{defn}

\begin{thm} (\cite{B86}, \cite{DG12}, \cite{M14})
Let $L$ be an even integral lattice.  There is a nontrivial central extension $\hat{L}$ of $L$ by $\langle \epsilon | \epsilon^2 = 1 \rangle$, unique up to isomorphism.  Furthermore, the rational Fock space $V_{L,\bQ} = \bQ\{\hat{L}\} \otimes_{\bQ} \Sym_{\bQ}(L \otimes t^{-1}\bQ[t^{-1}])$ admits a unique vertex algebra structure over $\bQ$, such that $Y(\iota(\gamma) \otimes 1,z) = E^-(-\gamma,z)E^+(-\gamma,z)\gamma z^\gamma$ for all $\gamma \in \hat{L}$.  Here, 
\begin{enumerate}
\item $\bQ\{\hat{L}\} = \bQ[\hat{L}]/(\epsilon + 1)$ denotes the twisted group ring,
\item $\iota(\gamma)$ is the image of $\gamma \in \hat{L}$ under the embedding into $\bQ\{\hat{L}\}$.  
\item $z^\gamma$ multiplies $\iota(\beta)$ by $z^{(\bar{\beta},\bar{\gamma})}$, where $\bar{\gamma}$ is the image of $\gamma$ in $L$.
\item $E^\pm(-\gamma,z) = \exp\left(\sum_{n \in \pm \bZ_{>0}} \frac{-\gamma(n)}{n}z^{-n} \right)$, where $\gamma(n)$ acts by multiplication by the element $\gamma(n) = \gamma \otimes t^n \in L \otimes t^{-1}\bQ[t^{-1}]$ when $n < 0$, takes $\iota(\beta)$ to $(\bar{\gamma},\bar{\beta})\iota(\beta)$ when $n=0$, and annihilates $\iota(\beta)$ and satisfies $[\gamma(n),\beta(m)] = n(\bar{\gamma},\bar{\beta})\delta_{n,-m}$ when $n > 0$.
\end{enumerate}
Moreover, $V_{L,\bQ}$ admits an integral form $V_L$, spanned by the coefficients of series of the form $E^-(\gamma,z)$ for $\gamma \in \hat{L}$, applied to vectors of the form $e^\gamma$ for $\gamma \in \hat{L}$, and if $L$ is unimodular, then $V_L$ admits a conformal vector.
\end{thm}

We note that by Proposition 5.8 in \cite{M14}, the integral form $V_L$ has a conformal vector if and only if $L$ is unimodular.

\begin{defn}
Let $(V, Y, \unit, \omega)$ be a conformal vertex algebra over a commutative ring $R$.  A symmetric $R$-bilinear form $(,)$ on $V$, with values in $R$, is \textbf{invariant} if
\[ (u,Y(u,z)w) = (Y(e^{L_1} (-z^2)^{L_0} u,z^{-1})v,w) \]
for all $u,v,w \in V$.
\end{defn}

\subsection{Integral forms for Virasoro and the Monster vertex algebra}

We briefly recall some previously known results.  Let $vir$ be the Lie algebra over $\bZ$ with basis $\{ L_i \}_{i \in \bZ} \cup \{ \frac{c}{2} \}$ and relations $[L_m,L_n] = (m-n)L_{m+n} + \binom{m+1}{3}\delta_{m+n,0} \frac{c}{2}$ and $[\frac{c}{2},L_n] = 0$.  We write $Witt_{>0}$ and $Witt_{<0}$ for the subalgebras spanned by $\{L_i\}_{i >0}$ and $\{L_i\}_{i < 0}$, respectively.

\begin{thm} (Theorem 5.7 of \cite{B99})
There exists a $\bZ$-Hopf subalgebra $U^+(vir)$ of $U(vir \otimes \bQ)$ whose Lie algebra of primitive elements is $vir$, and which admits a structural basis.  Furthermore, $U^+(vir)$ acts on the vertex algebra over $\bZ$ attached to any even unimodular lattice.
\end{thm}

We will not use the structural basis property, but it is essentially a smoothness property: when the Lie algebra of primitive vectors is finite dimensional, the existence of a structural basis is equivalent to being isomorphic to the Hopf algebra of differential operators on a smooth formal group.

\begin{defn}
We say that a vertex operator algebra over a commutative ring $R$ is \textbf{strongly conformal} if it admits an action of $U^+(vir) \otimes R$ compatible with the Virasoro action arising from $\omega$.
\end{defn}

\begin{thm}
There exists a vertex operator algebra $V^\natural_\bZ$ over $\bZ$, satisfying the following properties:
\begin{enumerate}
\item $V^\natural_\bZ$ is strongly conformal.
\item $V^\natural_\bZ$ has an invariant symmetric bilinear form, such that $V^\natural$ is self-dual.
\item $V^\natural_\bZ$ admits a faithful action of $\bM$ that preserves the bilinear form and commutes with the $U^+(vir)$-action.
\item $V^\natural_\bZ \otimes \bC$ is isomorphic to the Monster vertex operator algebra $V^\natural$ constructed in \cite{FLM88}.
\end{enumerate}
\end{thm}
\begin{proof}
All of the claims except the existence of the $U^+(vir)$ action are given in Theorem 3.22 in \cite{C17}.  The $U^+(vir)$ action is transported from the action on the Leech lattice vertex operator algebra $V_\Lambda$ by the cyclic orbifold construction in the following way: Let $P = \{2,3,5,7\}$.  Then for any pair of distinct $p,q \in P$, we constructed a $\bZ[1/pq,e^{\pi i/pq}]$-form of $V^\natural$, and in Theorem 3.9 of \textit{loc. cit.}  we showed that these forms have $\bM$-symmetry.  For each $p \in P$, there is a $p$B-pure elementary subgroup $H_p \subset \bM$ of order $p^2$, and we may decompose our $\bZ[1/pq,e^{\pi i/pq}]$-form into a direct sum of eigenspaces for the action of $H_p$.  By the cyclic orbifold correspondence, these eigenspaces are identified with eigenspaces for an action of an elementary subgroup on $V_\Lambda$, and thereby inherit the $U^+(vir) \otimes \bZ[1/pq,e^{\pi i/pq}]$-action.  For distinct $p,q,r \in P$, Proposition 3.15 of \textit{loc. cit.} gives an isomorphism between the $\bZ[1/pq,e^{\pi i/pq}]$-form and the $\bZ[1/pr,e^{\pi i/pr}]$-form, when both are base-changed to $\bZ[1/pqr,e^{\pi i/pqr}]$, and the actions of $U^+(vir) \otimes \bZ[1/pq,e^{\pi i/pq}]$ and $U^+(vir) \otimes \bZ[1/pr,e^{\pi i/pr}]$ are identified after base change.  Then, when $V^\natural_\bZ$ is constructed by descent, we obtain an action of $U^+(vir)$ by the same descent.
\end{proof}

\section{The integral no ghost theorem}

We apply an integral enhancement of the Goddard-Thorn \cite{GT72} no-ghost theorem from \cite{B99} to produce an integral form of the Monster Lie algebra.  



\begin{defn}
Let $R$ be a subring of $\bC$, and let $W$ be a free $R$-module with an action of the Virasoro algebra of central charge 26, equipped with an $R$-valued symmetric Virasoro-invariant bilinear form (where ``Virasoro-invariant'' means $L_i$ is adjoint to $L_{-i}$ for all integers $i$).  We define the ``old covariant quantization'' $OCQ_R(W)$ to be the $R$-module $P^1/N^1$ with its induced inner product, where $P^1 = \{ v \in W | L_0v = v, L_iv = 0, \forall i > 0 \}$ is the \textbf{weight 1 primary subspace}, and $N^1 = \{ v \in P^1 | (v,w) = 0, \forall w \in P^1 \}$ is the \textbf{nullspace} (or \textbf{radical}) of $P^1$. 
\end{defn}

The name ``old covariant quantization'' comes from string theory.  BRST cohomology is an equivalent functor, which is newer and considered ``more systematic'' \cite{P98}.  We use $OCQ$ because there is an existing treatment over rings other than $\bC$ in the literature \cite{B99}.

\begin{prop} \label{prop:basic-properties-of-OCW}
$OCQ_R$ describes a functor from the groupoid of Virasoro representations of central charge 26 with invariant inner product over $R$ to the groupoid of $R$-modules with nondegenerate inner product.  Here, morphisms in the source groupoid are inner-product preserving Virasoro-equivariant $R$-module isomorphisms, and morphisms in the target are inner-product preserving $R$-module isomorphisms.  If $W$ is a conformal vertex algebra and the inner-product is invariant in the sense of vertex algebras, then $OCQ_R(W)$ is a Lie algebra, and the induced inner product is invariant in the sense of Lie algebras.  Furthermore, vertex algebra automorphisms preserving the conformal element are taken to Lie algebra automorphisms.
\end{prop}
\begin{proof}
For the first claim, it suffices to show that any Virasoro-equivariant inner-product-preserving $R$-module isomorphism $V \to W$ induces an inner-product-preserving isomorphism of $R$-modules $P^1_V/N^1_V \to P^1_W/N^1_W$.  By Virasoro equivariance, this restricts to an Virasoro-equivariant inner-product-preserving $R$-module isomorphism $P^1_V \to P^1_W$, and because inner products are preserved, this map induces an $R$-module isomorphism $N^1_V \to N^1_W$ of nullspaces.  We then obtain an isomorphism of $R$-modules with nondegenerate inner product.

For the second claim, the Lie algebra structure is given by $[u+N^1_W,v+N^1_W] = (u_0 v) + N^1_W$.  The invariance claims can be checked after taking the tensor product with $\bC$, and they are well-known for $R = \bC$ (see e.g., Lemma 3.2.2 of \cite{C12}).
\end{proof}

We now consider the oscillator cancellation property.  Over $\bC$, if $V$ is a Virasoro representation of central charge 24 with finite dimensional $L_0$-eigenspaces and nondegenerate Virasoro-invariant inner product, and $\beta \in (I\!I_{1,1} \otimes \bC) \setminus \{0\}$, then $OCQ(V \otimes \pi^{1,1}_\beta)$ is isomorphic to the eigenspace $V_{1-\frac{(\beta,\beta)}{2}} = \{ v \in V | L_0v = v - \frac{(\beta,\beta)}{2}v \}$.  Here, $\pi^{1,1}_\beta$ is the irreducible module for the Heisenberg vertex algebra $\pi^{1,1}_0$ attached to the vector $\beta$.

\begin{defn}
Let $R$ be a subring of $\bC$, and let $V$ be a $U^+(vir)_R$-module equipped with a symmetric nondegenerate Virasoro-invariant bilinear form.  We fix vectors $\beta, \gamma \in I\!I_{1,1} \otimes R$ satisfying $(\beta,\gamma) \neq 0, (\gamma,\gamma) = 0$, let $\pi^{1,1}_{\beta,R}$ be the Heisenberg module attached to $\beta$ over $R$, and define $H = V \otimes_R \pi^{1,1}_{\beta,R}$.  We define operators $K_i = \gamma(-i) = (1 \otimes \gamma(-1))_i$ on $H$, and define the \textbf{transverse subspace} to be $T = \{ v \in H | K_i v = L_i v = 0, \forall i > 0\}$.
\end{defn}


\begin{thm} \label{thm:self-duality-for-transverse-space}
Let $R$ be a subring of $\bC$, and let $V$ be an $R$-free $U^+(vir)_R$-module equipped with an action of a group $G$ that commutes with the Virasoro action, and with a symmetric nondegenerate Virasoro-invariant $G$-invariant bilinear form such that $V$ is self-dual with respect to the form.  Then, with notation as given in the preceding definition, if $(\beta,\gamma) \in R^\times$, then $V \cong T$ as $R$-free $R[G]$-modules with inner product, and in particular $T$ is self-dual.
\end{thm}
\begin{proof}
The proof over $\bC$ works by considering the subspace $K_\bC = \{ x \in H_\bC | K_i x = 0, \forall i > 0 \}$ of $H_\bC$, and showing that the radical of the inner product on $K_\bC$ has both $V_\bC \otimes \bC e^\beta$ and $T_\bC$ as complements in $K_\bC$.  We then obtain an inner-product-preserving isomorphism between $V_\bC$ and $T_\bC$, i.e., an inner-product-preserving map $V_\bC \to K_\bC$ that takes values in elements $x$ satisfying $L_i x = 0$ for all $i > 0$.  To show that this isomorphism exists over $R$, it suffices to produce an $R$-linear map $V \to K = \{ x \in H | K_i x = 0, \forall i > 0 \}$ that takes each $v \in V$ to an element in $K$ of the form $v \otimes e^\beta + \cdots$, that is annihilated by all $L_i, i > 0$, where the unwritten summands lie in the radical.  

The argument of \cite{B99} section 6 adapts to our situation, but we need to change the language to remove the assumption that we are analyzing the fake Monster Lie algebra.  Within the Heisenberg vertex algebra $\pi^{1,1}_0$, we have a commutative $R$-algebra $Y^-$ generated by $\{ e^{-\gamma}_{-1}D^{(n)}e^\gamma\}_{n > 0}$, and a subalgebra $Y'$ generated by $K_i$, $i < 0$.  These two algebras become equal after passing to the field of fractions.  The action of $U^+(vir)$ on $H$ restricts to actions of the subalgebra $U = U^+(Witt_{>0})$ and the subalgebra $U' \subset U$ generated by $L_i$, $i > 0$.  As with $Y^-$ and $Y'$, we see that $U$ and $U'$ become equal after base change to the field of fractions.  We know from the lattice vertex algebra construction that the subspace $Y^- e^\beta \subset \pi^{1,1}_{\beta,R}$ is stable under the action of $U^+(vir)$, so $U$ and $U'$ also act on this subspace.

In the proof of Lemma 6.3 of \cite{B99}, we find that the graded dual of $Y^- e^\beta$ contains the free $U$-module generated by $e^{\beta,*}$ as a sub-$U$-module, and the subspace $(e^{\beta,*}U)_n$ lies in $(Y^-e^\beta)^*_n$ as an index $\prod_{|\lambda|=n} (\beta,\gamma)^{l(\lambda)}$ subgroup, where $l(\lambda)$ is the number of parts in the partition $\lambda$.  Thus, by our assumption that $(\beta,\gamma)$ is invertible in $R$, these subspaces are equal.


By freeness of $e^{\beta,*}U$, for any $v \in V$, there is a unique $U$-module map $e^{\beta,*}U \to V$ taking $e^{\beta,*}$ to $v$.  This induces an isomorphism $V \to \Hom_U(e^{\beta,*}U,V)$ of $R$-free $R[G]$-modules.  Dualizing the rank 1 free $U$-module, we have the isomorphism
\[ \Hom_U(e^{\beta,*}U, V) \simto \Hom_U(R, V \otimes_R Y^- e^\beta). \]
Thus, composing these two isomorphisms yields an inner-product-preserving isomorphism from $V$ to the space of $U$-invariant elements of $V \otimes_R Y^- e^\beta$.  That is, for each $v \in V$, there is a unique $U$-invariant vector in $V \otimes_R Y^- e^\beta$ of the form $v \otimes e^\beta + \cdots$, where the unwritten summands are elements of $V \otimes I e^\beta$, with $I$ denoting the augmentation ideal of $Y^-$.  By $U$-invariance, we know that this vector is annihilated by $L_i$ for all $i>0$.

It remains to show that the vectors of the form $v \otimes e^\beta + \cdots$ lie in $K$, rather than just $K_\bC$.  Since $\pi^{1,1}_\beta$ is $U^+(vir)$-stable, such vectors lie in $H$, so it suffices to show that they are annihilated by all $K_i, i>0$.  However, this can be checked after base change to $\bC$, where it is already known to hold.  Thus, we have an injective map $V \to T$ of $R$-free $R[G]$-modules that preserves the inner product.  This is an isomorphism, because the base change $V_\bC \to T_\bC$ is an isomorphism, and $V$ is self-dual.


The inverse map $T \to V$ is defined by taking the $e^\beta$ term of any vector $v \otimes e^\beta + \cdots$ in $T$ to get a vector in $V$.  Any additional terms lie in the radical of the form on $K$.
\end{proof}

To complete the proof of oscillator cancellation, we show that the primary space of weight 1 is the direct sum of the weight 1 transverse space with the null space.  We are thankful to Richard Borcherds for pointing out the importance of self-duality of $T$.

\begin{lem} \label{lem:split-projection-for-self-dual-modules}
Let $R$ be a subring of $\bC$, let $P$ be a finite rank free $R$-module with $R$-valued symmetric bilinear form, let $T$ be a submodule, and let $N$ be the radical of the form.  Suppose $T$ is self-dual with respect to the form, and $P \otimes_R \bC = T \otimes_R \bC \oplus N \otimes_R \bC$.  Then $P = T \oplus N$.
\end{lem}
\begin{proof}
We first note that there is a unique inner-product-preserving projection $P \otimes_R \bC \to T \otimes_R \bC$, namely the projection that has kernel $N \otimes_R \bC$.  Since $T \subseteq P$, the image of $P$ under this projection necessarily contains $T$, and since the bilinear form on $P$ is $R$-valued, the image is contained in the dual module $T^\vee = \{x \in T \otimes_R \bC | \forall y \in T, (x,y) \in R \}$.  By self-duality of $T$, restriction yields a surjective norm-preserving projection $P \to T$, and the kernel is $N$.
\end{proof}

\begin{thm} \label{thm:no-ghost}
Let $R$ be a subring of $\bC$, and let $V$ be a unitarizable $U^+(vir)_R$-module of half central charge 12 equipped with an action of a group $G$ that commutes with the Virasoro action, and with a Virasoro-invariant $G$-invariant bilinear form, such that $V$ is self-dual.  Let $\beta$ be an element of $I\!I_{1,1} \otimes R$ such that $(\beta, \gamma) \in R^\times$ for some $\gamma \in I\!I_{1,1} \otimes R$, let $\pi^{1,1}_\beta$ be the Heisenberg module attached to $\beta$, and let $H = V \otimes \pi^{1,1}_\beta$.  Define $P^1 = \{ v \in H | L_0 v = v, L_i v = 0, \forall i > 0 \}$, and $N^1$ the radical of the inner product on $P^1$.  Then, $V_{1-\frac{(\beta,\beta)}{2}} \cong P^1/N^1$ as $R[G]$-modules.
\end{thm}
\begin{proof}
Let $\gamma\in I\!I_{1,1} \otimes R$ satisfy $(\beta, \gamma) \in R^\times$.

\noindent{\textbf{Case 1:}} If $\beta$ has nonzero norm, we may take $\gamma$ to be a sum of linearly independent norm zero vectors $\gamma_1, \gamma_2$, with $(\beta,\gamma_i) \neq 0$.  Let $R_1 = R[(\beta,\gamma_1)^{-1}]$, $R_2 = R[(\beta,\gamma_2)^{-1}]$, and $R_3 = R[(\beta,\gamma_1)^{-1},(\beta,\gamma_2)^{-1}]$.  Let $T_1$ be the transverse space over $R_1$ constructed using $\gamma_1$, and $T_2$ the transverse space over $R_2$ constructed using $\gamma_2$.  By Theorem \ref{thm:self-duality-for-transverse-space}, for $i=1,2$, each weight 1 subspace $T_i^1$ is self-dual, and by the last half of the Goddard-Thorn theorem over $\bC$, $P^1 \otimes_R \bC = (T_i^1 \otimes_{R_i} \bC) \oplus (N^1 \otimes_R \bC)$.  Thus, we may apply Lemma \ref{lem:split-projection-for-self-dual-modules} to conclude $P^1 \otimes_R R_i = T_i^1 \oplus (N^1 \otimes_R R_i)$.  Once again applying Theorem \ref{thm:self-duality-for-transverse-space}, we obtain the isomorphism $(V \otimes_R R_i e^\beta)^1 \cong (P^1/N^1) \otimes_R R_i$, and $(V \otimes_R R_i e^\beta)^1 = V_{1-\frac{(\beta,\beta)}{2}} \otimes_R R_i e^\beta$ by degree considerations.

We now apply Zariski descent.  We have $R_i[G]$-module isomorphisms $V_{1-\frac{(\beta,\beta)}{2}} \otimes_R R_i \cong (P^1/N^1) \otimes_R R_i$ for $i=1,2$, and they glue along $R_3$ to yield an isomorphism $V_{1-\frac{(\beta,\beta)}{2}} \cong P^1/N^1$ over $R$.

\noindent{\textbf{Case 2:}} If $\beta$ has norm zero, then there exists a norm zero $\gamma_1$ such that $(\beta,\gamma_1) = (\beta,\gamma) \in R^\times$.  In this case, descent is unnecessary, and the first half of the above argument yields $V_{1-\frac{(\beta,\beta)}{2}} \cong P^1/N^1$ as $R[G]$-modules.
\end{proof}

\begin{rem}
Theorem \ref{thm:no-ghost} is our refinement of the integral no-ghost theorem (Theorem 6.5 of \cite{B99}).  Borcherds's theorem is only stated for the case $V$ is the vertex algebra of a Niemeier lattice, and it only asserts that if $\beta$ is $n$ times a primitive vector, then $OCQ_{\bZ}(H)$ has determinant dividing a power of $n$.  However, his proof is substantially more general, and implicitly uses the isomorphism that we establish.
\end{rem}

\begin{rem}
In \cite{B99}, Borcherds remarks that the root spaces of some Lie algebras that come from $OCQ$ do not depend solely on the norm of the root.  As an example, he points out that for $\beta \in I\!I_{25,1}$ primitive of norm zero, the degree $\beta$ piece of the fake Monster Lie algebra is isomorphic to the Niemeier lattice $\beta^\perp/\bZ\beta$.  This does not contradict the isomorphism in our result for the following reason: to isolate the degree $\beta$ piece in the norm zero part of a $I\!I_{1,1}$-grading, we must take $V$ to be the lattice vertex algebra $V_{\beta^\perp/\bZ\beta}$ in the statement of our theorem.  For other primitive norm zero vectors, we need to input other lattice vertex algebras as $V$.
\end{rem}

\begin{defn} \label{defn:R-form-of-monster-Lie-algebra}
Let $R$ be a subring of $\bC$.  We define the $R$-form $\fm_R$ of the Monster Lie algebra to be $OCQ_R(V^\natural_R \otimes_R V_{I\!I_{1,1},R})$.
\end{defn}

\begin{cor} \label{cor:description-of_monster-Lie-algebra}
Let $R$ be a subring of $\bC$.  The $R$-form of the Monster Lie algebra $\fm_R$ is a $I\!I_{1,1}$-graded Lie algebra with homogeneous action of $\bM$, and the $\bM$-action on the graded pieces $\fm_{m,n,R}$ is given by the following $R[\bM]$-module isomorphisms
\[ \fm_{m,n,R} \cong \begin{cases} V^\natural_{1+mn,R} & \gcd(m,n) \in R^\times \\ I\!I_{1,1} \otimes R \text{ with trivial action} & (m,n) = (0,0) \\ 0 & mn = 0, (m,n) \neq (0,0) \\ \text{mysterious} & \text{otherwise}\end{cases} \]
Furthermore, for any subgroup $G$ of $\bM$ and any ring homomorphism $\phi: \Rep^\natural_R(G) \to \bC$, if $\gcd(m,n)$ is a unit in $R' = \bigcap_{P \supset |G|R} R_P$, then
\[ \phi(\fm_{m,n,R}) = \phi(V^\natural_{1+mn,R}). \]
Moreover, $\fm_R$ has an invariant bilinear form that restricts to a perfect pairing between the degree $(m,n)$ and $(-m,-n)$ submodules, and a contragradient involution $\omega$ that negates degrees in $I\!I_{1,1}$, such that the homogeneous pieces with lattice degree $(m,n)$ satisfying $\gcd(m,n) \in R^\times$ are self-dual under the corresponding contragradient form.
\end{cor}
\begin{proof}
The Lie algebra structure, invariant form, and Monster action come from basic properties of the $OCQ$ functor.  The grading and involution come from the corresponding structures on $V_{I\!I_{1,1},R}$.  The identification of $\fm_{m,n,R}$ with homogeneous spaces of $V^\natural$ follows from Theorem \ref{thm:no-ghost}.  The cases where $mn = 0$ follow from the decomposition of the space of weight one vectors in $V^\natural_R \otimes \pi^{1,1}_{(m,n),R}$ as the direct sum of $V^\natural_{1,R} \otimes_R \pi^{1,1}_{(m,n),R,0} = 0$ and $V^\natural_{0,R} \otimes_R \pi^{1,1}_{(m,n),R,1} = R \otimes_R (I\!I_{1,1} \otimes R)$, and the fact that $L_1\gamma^{-1}e^\beta = (\gamma,\beta)e^\beta$.  This means the primary vectors are $\gamma_{-1}e^\beta$ for $(\gamma,\beta) = 0$, and the inner product vanishes identically on this space except when $\beta = 0$, in which case it is nondegenerate.  The equality of $\phi$-values follows from Corollary \ref{cor:homs-factor-through}.
\end{proof}

The reader might wonder if we can have stronger control over the parts of $\fm_\bZ$ in non-primitive lattice degree $\beta$.  As the next result shows, there seems to be a fundamental problem with cancelling $\pi^{1,1}_\beta$. 

\begin{prop}
Let $\beta \in I\!I_{1,1}$ be a non-primitive vector of negative norm.  Then, there exists a Virasoro representation $V$ of central charge 24 over $\bZ$, such that $P^1/N^1 \not\cong V_{1-\frac{(\beta,\beta)}{2}}$.
\end{prop}
\begin{proof}
Let $V = U^+(Witt_{<0})v$ for a nonzero vector $v$ of conformal weight $-(\beta,\beta)/2 > 0$.  This is an integral form of the Verma module $M(24,-\frac{(\beta,\beta)}{2})$.  Then the lowest weight spaces are $\bZ v$ of weight $-\frac{(\beta,\beta)}{2}$ and $\bZ L_{-1}v$ of weight $1-\frac{(\beta,\beta)}{2}$.  Because $e^\beta$ has weight $\frac{(\beta,\beta)}{2}$, the weight 1 subspace of $V \otimes \pi^{1,1}_\beta$ is then $(I\!I_{1,1} \otimes \bZ v) \oplus (\bZ e^\beta \otimes \bZ L_{-1}v)$, where we write $I\!I_{1,1}$ to mean $\{ \gamma_{-1}e^\beta | \gamma \in I\!I_{1,1} \}$, which lies in weight $1+\frac{(\beta,\beta)}{2}$.  $L_1$ acts on this space by $L_1(\gamma_{-1}e^\beta \otimes n v + r e^\beta \otimes sL_{-1}v) = (\gamma,\beta)e^\beta \otimes nv - r e^\beta \otimes s(\beta,\beta)v$, so $P^1$ consists of all $\gamma_{-1}e^\beta \otimes n v + r e^\beta \otimes sL_{-1}v$ satisfying $n(\gamma,\beta) = rs(\beta,\beta)$.  This is spanned by $\beta_{-1}e^\beta \otimes v + e^\beta \otimes L_{-1}v$ and $\gamma_{-1} e^\beta \otimes v$ for $(\gamma,\beta) = 0$.  $N^1$ is spanned by $L_{-1}(e^\beta \otimes v) = \beta_{-1} e^\beta \otimes v + e^\beta \otimes L_{-1}v$, so the quotient $P^1/N^1$ is represented by $\{\gamma_{-1} e^\beta \otimes v | (\gamma,\beta) = 0\}$ with its induced inner product.  Let us compare this space with $V_{1-\frac{(\beta,\beta)}{2}}$, i.e., the span of $L_{-1} v$.  Let $(v,v) = a$ for $a \neq 0$, so $(L_{-1}v,L_{-1}v) = (L_1 L_{-1}v,v) = -(\beta,\beta)(v,v) = -(\beta,\beta)a$.  On the other hand, $(\gamma_{-1} e^\beta,\gamma_{-1}e^\beta) = (\gamma,\gamma)$.  Thus, $\gamma_{-1} e^\beta \otimes v$ has norm $(\gamma,\gamma)a$.  The subgroup of $P^1/N^1$ generated by vectors of norm $-(\beta,\beta)a$ is therefore strictly smaller than $P^1/N^1$ when $\beta$ is not primitive.
\end{proof}


\section{Hodge theory on Lie algebra cohomology}

We introduce Borcherds-Kac-Moody Lie algebras over subrings $R$ of $\bC$, and extend a theorem of Kostant about the spectrum of the Laplacian to this setting.  We roughly follow the treatment in \cite{J96}.

\begin{defn}
Let $I$ be a countable set, and let $A = (a_{i,j})_{i,j \in I}$ be a real matrix.  We say that $A$ is a \textbf{Borcherds-Cartan matrix} if the following 3 conditions are satisfied:
\begin{enumerate}
\item $A$ is symmetrizable, i.e., there is a diagonal matrix $D$ whose diagonal entries are positive real, such that $DA$ is symmetric.
\item If $a_{i,i} > 0$, then $a_{i,i} = 2$.
\item For all $k \neq j$, $a_{j,k} \leq 0$.  For all $i \in I$ such that $a_{i,i} > 0$, $a_{i,j} \in \bZ_{\leq 0}$ for all $j \neq i$.
\end{enumerate}
Let $A$ be a Borcherds-Cartan matrix.  The Borcherds-Kac-Moody Lie algebra $\fg(A)$ is the Lie algebra over $\bC$ with generators $e_i, f_i, h_i, i \in I$, and the following defining relations:
\begin{enumerate}
\item For all $i,j,k \in I$, $[h_i,h_j] = 0$, $[h_i,e_k] = a_{i,k}e_k$, $[h_i,f_k] = -a_{i,k}f_k$, $[e_i, f_j] = \delta_{i,j}h_i$
\item For all $i \neq j$ with $a_{i,i} > 0$, $(\ad e_i)^{-2\frac{a_{i,j}}{a_{i,i}}+1}e_j = 0$, $(\ad f_i)^{-2\frac{a_{i,j}}{a_{i,i}}+1}f_j =0$.  For all $i,j \in I$ such tha $a_{i,j} = 0$, $[e_i,e_j]=0$ and $[f_i,f_j]=0$.
\end{enumerate}
A \textbf{Borcherds-Kac-Moody Lie algebra} (also called \textbf{BKM algebra}, or \textbf{generalized Kac-Moody Lie algebra}) over $\bC$ is a Lie algebra $\fg$ isomorphic to one of the form $\fg(A)/C \rtimes D$, where $C$ is a central ideal in $\fg(A)$, and $D$ is a commutative Lie algebra of derivations.  The \textbf{Cartan subalgebra} of $\fg(A)/C \rtimes D$ is $\fh = (\fh(A)/C) \rtimes D$, where $\fh(A)$ is the subalgebra of $\fg(A)$ generated by $\{h_i\}$.
\end{defn}

\begin{lem}
Let $\fg$ be a Borcherds-Kac-Moody Lie algebra over $\bC$.  Then $\fg$ admits an invariant bilinear form, and this is unique up to rescaling.  
\end{lem}
\begin{proof}
The existence and uniqueness is Theorem 1.6 of \cite{J96}.  
\end{proof}


\begin{defn}
Let $\fg$ be a BKM algebra over $\bC$.  A \textbf{root} is a nonzero eigenvalue for the action of $\fh$, i.e., a linear map $\alpha: \fh \to \bC$ such that there exists a nonzero $v \in \fg$ such that for all $h \in \fh$, $[h,v] = \alpha(h) v$.  The set of roots is written $\Delta$.  The \textbf{root space} attached to a root $\alpha$ is the eigenspace $\fg_\alpha = \{ v \in \fg | \forall x \in \fh, [x,v] = \alpha(x)v\}$.  A \textbf{simple root} is a root given by the eigenvalue $\alpha_i$ of $e_i$ for some $i$ (characterized by $\alpha_i(h_j) = a_{i,j}$).  A root $\alpha$ is \textbf{real} (written $\alpha \in \Delta^{re}$) if $(\alpha,\alpha) > 0$, and \textbf{imaginary} (written $\alpha \in \Delta^{im}$) otherwise.  The \textbf{root lattice} is the $\bZ$-span of simple roots in $\fh^*$.  A root $\alpha$ is \textbf{positive} (written $\alpha \in \Delta_+$) if it is a non-negative integer combination of $\alpha_i$, and \textbf{negative} if $-\alpha$ is positive. 
A \textbf{Weyl vector} for a BKM algebra over $\bC$ is some $\rho \in \fh^*$ such that $(\rho,\alpha) = -(\alpha,\alpha)/2$ for all simple roots $\alpha$.
\end{defn}

We briefly review the properties of the Monster Lie algebra $\fm$.

\begin{thm} (\cite{B92})
The Monster Lie algebra $\fm$ is a BKM algebra over $\bC$, with 2-dimensional Cartan subalgebra $\fh$, and roots graded by $\bZ \times \bZ$ with inner product $((a,b),(c,d)) = ad+bc$ (that is, the lattice $I\!I_{1,1}$).  The multiplicity of the $(m,n)$ root space is $c(mn)$, where $J(\tau) = \sum_{n \geq -1} c(n) q^n = q^{-1} + 196884q + 21493760q^2 + \cdots$.  The positive roots are those of degree $(m,n)$ with $m > 0$, and the simple roots are those of degree $(1,n)$ for some $n$.  The Weyl vector is $(-1,0)$.  When $(m,n) \neq (0,0)$, the degree $(m,n)$ part of $\fm$ is isomorphic to $V^\natural_{1+mn}$ as a representation of $\bM$.  In particular, the isomorphism type depends only on the product $mn$.
\end{thm}

\begin{defn}
Let $R$ be a subring of $\bC$.  A BKM algebra over $R$ is a Lie algebra $\fg_R$ over $R$ such that the following properties hold:
\begin{enumerate}
\item $\fg_R \otimes _R \bC$ is a BKM algebra over $\bC$.
\item $\fg_R$ has a subalgebra $\fh_R$ such that $\fh_R \otimes_R \bC$ is a Cartan subalgebra of $\fg_R \otimes _R \bC$, and the roots of $\fg_R \otimes _R \bC$ (with respect to $\fh_R \otimes_R \bC$) restrict to $R$-linear functionals $\fh_R \to R$.
\item $\fg_R$ splits as a direct sum of $\fh_R$ and root spaces.
\end{enumerate}
\end{defn}

\begin{thm} \label{thm:monster-Lie-is-a-BKM-over-Z}
The Monster Lie algebra $\fm_\bZ$ is a BKM algebra over $\bZ$.  In particular, any $R$-form $\fm_R$ is a BKM algebra over $R$.
\end{thm}
\begin{proof}
From Corollary \ref{cor:description-of_monster-Lie-algebra}, we obtain the decomposition into root spaces, and also see that $\fm_R = \fm_\bZ \otimes R$, so in particular, if $R = \bC$, we obtain $\fm$.  The Cartan subalgebra is $\fh_\bZ = I\!I_{1,1}$, and the roots are $\bZ$-linear functionals to $\bZ$.
\end{proof}

\subsection{The Chevalley-Eilenberg complex}

\begin{defn}
Given a Lie algebra $\fn$ over a subring $R$ of $\bC$, the \textbf{Chevalley-Eilenberg complex} for the trivial $\fn$-module $R$ is $(\bigwedge^*\fn, \delta)$, where
\[ \delta(x_1 \wedge \cdots \wedge x_n) = \sum_{i<j} (-1)^{i+j+1} [x_i,x_j]\wedge x_1 \wedge \cdots \wedge \hat{x}_i \wedge \cdots \wedge \hat{x}_j \wedge \cdots \wedge x_n. \]
The homology of $\fn$ with coefficients in $R$ is the homology of the Chevalley-Eilenberg complex.
\end{defn}

We note that in particular, $\delta(x \wedge y) = [x,y]$ for $x,y \in \fn$.

\begin{lem} \label{lem:Hodge-structure-on-BKM}
Let $\fg$ be a BKM algebra over a subring $R$ of $\bC$ with a positive definite contravariant hermitian form $\langle, \rangle$, and let $\fn$ be the positive subalgebra, i.e., the subalgebra generated by $\{e_i \}_{i \in I}$.  Then:
\begin{enumerate}
\item The operator $\delta$ satisfies
\[ \begin{aligned} \delta(a_1 \wedge \cdots \wedge a_r) &= \sum_{1\leq s<t\leq r} \epsilon(s,t) \delta(a_s \wedge a_t) \wedge a_1 \wedge \cdots \hat{a}_s \cdots \hat{a}_t \cdots \wedge a_r \\
&\qquad + (r-2) \sum_{1 \leq s \leq r} (-1)^{d_s(d_1+\cdots+d_{s-1})} \delta(a_s) \wedge a_1 \wedge \cdots \hat{a}_s \cdots \wedge a_r
\end{aligned} \]
for $a_i \in \bigwedge^{d_i}\fn$, and $\epsilon(s,t) = (-1)^{(d_s+d_t)(d_1+\cdots+d_{s-1}) + d_t(d_{s+1}+\cdots+d_{t-1})}$.  In particular, $\delta$ is a differential operator of order 2 and degree $-1$.
\item $\delta$ has an adjoint $d$ on $\bigwedge^* \fn$ with respect to the contravariant form.
\item $d$ and $\delta$ are disjoint square-zero operators that respect the weight grading.  Here, by ``disjoint'', we mean $d\delta x = 0$ implies $\delta x = 0$ and $\delta dx = 0$ implies $dx = 0$.
\item $d(a \wedge b) = da \wedge b + (-1)^{ij} a \wedge db$ for $a \in \bigwedge^i \fn$, $b \in \bigwedge^j \fn$.  In particular, $d$ is a differential operator of order 1 and degree 1.
\end{enumerate}
\end{lem}
\begin{proof}
The first claim is the lemma after Theorem 2.5.1 in \cite{F86}, and follows from a direct calculation.

The second claim is clear from the definition of contravariant form.

For the third claim, the square-zero property is clear, and the preservation of the weight grading follows from the explicit formula for $\delta$.  Disjointness follows from the positive definite property of the contravariant inner product on $\bigwedge^* \fn$.  Specifically, if $d\delta x = 0$, then $(\delta x, \delta x) = (d\delta x, x) = 0$, implying $\delta x = 0$, and similarly for $dx$.  This was noted in Remark 2.3 in \cite{K61}.

The fourth claim is given in the same place as the first.
\end{proof}

\begin{lem} \label{lem:d-applied-to-bracket-with-ei}
Let $\fg$ be a BKM algebra over a subring $R$ of $\bC$ with Weyl vector $\rho$, and a nondegenerate contravariant hermitian form $\langle, \rangle$.  Let $\fn$ be the positive subalgebra.  Then, for any positive root $\gamma$, $d[e_i,x] = \gamma(h_i) e_i \wedge x +  e_i \cdot dx$ for all $x \in \fn_\gamma$.
\end{lem}
\begin{proof}
We will use the contravariant inner product on $\bigwedge^* \fg$, restricted to $\bigwedge^* \fn$ to do our computations.  In particular, we compute on one hand $\langle d[e_i,x], a \wedge b \rangle$ and on the other hand $\langle e_i \wedge x , a \wedge b \rangle + \langle e_i \cdot dx, a \wedge b \rangle$.  Because the inner product on $\bigwedge^2 \fn$ is nondegenerate, equality after taking the inner product with arbitrary $a \wedge b$ implies equality of the original elements.  However, we need to be careful with contravariance, because applying the Chevalley generator $f_i$ may produce an element of $\fg$ not in $\fn$.  Most notably, if $x, y \in \fn$, then the equality $\langle x, [h_i,y] \rangle = \langle dx, h_i \wedge y \rangle$ holds in $\bigwedge^* \fg$ but $h_i \not\in \fn$, so the right side does not make sense as an inner product in $\bigwedge^* \fn$.

We split the computation into cases, using the fact that as a subspace of $\fg$, $\fn = (\ad e_i)\fn + \ker (\ad f_i) + \bC e_i$.  Thus, when taking the inner product with $a \wedge b$, we may assume that $a$ and $b$ are either taken by $\ad f_i$ to $\fn$, or are equal to $e_i$.  Our first computation is:
\[ \begin{aligned}
\langle d[e_i,x], a \wedge b \rangle &= \langle [e_i,x], [a,b] \rangle \\
&= -\langle x, [f_i,[a,b]] \rangle \\
&= \langle x, [a,[b,f_i]] \rangle + \langle x, [b,[f_i,a]] \rangle,
\end{aligned} \]
regardless of the form of $a$ and $b$.  We write $dx = \sum a_j \wedge b_j$.  Then, if $[a,f_i], [b,f_i] \in \fn$, we have
\[ \begin{aligned} \langle e_i \cdot dx, a \wedge b \rangle &= \sum_j \langle [e_i,a_j],a \rangle\langle b_j,b \rangle - \langle [e_i,a_j],b \rangle\langle b_j,a \rangle \\
&\qquad + \langle a_j, a \rangle \langle [e_i,b_j], b \rangle - \langle a_j,b \rangle\langle [e_i,b_j],a \rangle \\
&= \sum_j \langle a_j, [a,f_i] \rangle \langle b_j, b \rangle - \langle a_j, [b,f_i] \rangle \langle b_j, a \rangle \\
&\qquad + \langle a_j, a \rangle \langle b_j, [b,f_i] \rangle - \langle a_j,b \rangle\langle b_j, [a,f_i] \rangle \\
&= \langle dx, [a,f_i] \wedge b + a \wedge [b,f_i] \rangle \\
&= \langle x, [a,[b,f_i]] \rangle + \langle x, [b,[f_i,a]] \rangle.
\end{aligned} \]
Thus, in these cases, we have equality
\[ \langle d[e_i,x], a \wedge b \rangle = \langle e_i \cdot dx, a \wedge b \rangle. \]
Furthermore, in these cases, $\langle e_i \wedge x , a \wedge b \rangle = \langle e_i, a \rangle \langle x, b \rangle - \langle e_i,b \rangle\langle x,a \rangle = 0$, so our claim is satisfied in this case.

If $a = e_i$ and $[f_i,b] \in \fn$, with $b \in \fg_\beta$, then our 3 terms are
\[ \begin{aligned}
\langle d[e_i,x], a \wedge b \rangle &= \langle x, [a,[b,f_i]] \rangle + \langle x, [b,[f_i,a]] \rangle \\
&= \langle x, [a,[b,f_i]] \rangle + \langle x, [h_i,b] \rangle \\
&= \langle x, [a,[b,f_i]] \rangle + \beta(h_i) \langle x, b \rangle
\end{aligned} \]
and 
\[ \begin{aligned} \langle e_i \cdot dx, a \wedge b \rangle &= \sum_j \langle [e_i,a_j], e_i \rangle\langle b_j,b \rangle - \langle [e_i,a_j],b \rangle\langle b_j, e_i \rangle \\
&\qquad + \langle a_j, e_i \rangle \langle [e_i,b_j], b \rangle - \langle a_j,b \rangle\langle [e_i,b_j],e_i \rangle \\
&= \sum_j \langle a_j, e_i \rangle \langle b_j, [b,f_i] \rangle - \langle a_j,[b,f_i] \rangle\langle b_j, e_i \rangle \\
&= \langle dx, a \wedge [b,f_i]\rangle \\
&= \langle x, [a,[b,f_i]] \rangle
\end{aligned} \]
and
\[ \begin{aligned} \langle e_i \wedge x , a \wedge b \rangle &= \langle e_i, e_i \rangle \langle x, b \rangle - \langle e_i,b \rangle\langle x, e_i \rangle \\
&= \langle x, b \rangle 
\end{aligned} \]
so our claim is satisfied in this case, as well.  The case $[f_i,a] \in \fn$ and $b = e_i$ is essentially the same computation, so we omit it, and the case $a = b = e_i$ trivially yields $a \wedge b = 0$.  This completes the proof.
\end{proof}

\begin{lem} \label{lem:d-applied-to-any-bracket}
Let $x \in \fg_\alpha$ and $y \in \fg_\beta$.  Then, $d[x,y] = (\alpha,\beta) x\wedge y + x \cdot dy - y \cdot dx$.
\end{lem}
\begin{proof}
We induct on the number of simple roots into which $\alpha$ decomposes.  If $\alpha$ is simple, then we may assume $x$ is some Chevalley generator $e_i$, and then this is resolved by Lemma \ref{lem:d-applied-to-bracket-with-ei}.  If $x = [e_i,x']$, then $[x,y] = [[e_i,x'],y] = [e_i,[x',y]] - [x',[e_i,y]]$, so we have reduced to a sum of brackets where the left entry decomposes into fewer simple roots.  Assuming the claim holds in those cases, we have
\[ \begin{aligned}
d[x,y] &= d[e_i,[x',y]] - d[x',[e_i,y]] \\
&= (\alpha_i,\alpha-\alpha_i+\beta)e_i \wedge[x',y] + e_i \cdot d[x',y] - (\alpha-\alpha_i, \beta + \alpha_i)x' \wedge [e_i,y] \\
&\qquad - x' \cdot d[e_i,y] + [e_i,y] \cdot dx' \\
&= (\alpha_i,\alpha-\alpha_i+\beta)e_i \wedge [x',y] + (\alpha-\alpha_i,\beta)(x \wedge y + x' \wedge [e_i,y]) + e_i \cdot x'\cdot dy \\
&\qquad - e_i \cdot y \cdot dx' - (\alpha-\alpha_i, \beta + \alpha_i)x' \wedge [e_i,y] + (\alpha_i,\beta)(x \wedge y - e_i \wedge [x',y]) \\
&\qquad - x' \cdot e_i \cdot dy + [e_i,y] \cdot dx' \\
&= (\alpha_i, \alpha - \alpha_i)e_i \wedge [x',y] + (\alpha,\beta)x\wedge y + (\alpha-\alpha_i,-\alpha_i)x' \wedge [e_i,y] \\
&\qquad + x \cdot dy - y \cdot e_i \cdot dx'
\end{aligned} \]
Using Lemma \ref{lem:d-applied-to-bracket-with-ei}, we see that:
\[\begin{aligned}
y \cdot dx &= y \cdot d[e_i,x'] \\
&= (\alpha_i, \alpha -\alpha_i)([y,e_i] \wedge x' + e_i \wedge [y,x']) + y \cdot e_i \cdot dx'
\end{aligned} \]
so combining this with the previous computation, we obtain the answer we want.
\end{proof}

\subsection{Properties of the Laplacian}

The following result is stated in \cite{B98} section 3, but we have been unable to find a proof in the generality we need here.  In the finite dimensional case, it is Theorem 5.7 in \cite{K61}.

\begin{thm} \label{thm:Laplacian}
Let $R$ be a subring of $\bC$ that is closed under complex conjugation, let $\fg$ be a BKM algebra over $R$ with Weyl vector $\rho$, and suppose $\fg$ is given a nondegenerate contravariant hermitian form.  Let $\fn$ be the positive subalgebra.  Let $(\bigwedge^*(\fn), \delta)$ be the Chevalley-Eilenberg complex for the trivial $\fn$-module $R$, and let $d$ be the adjoint of $\delta$ with respect to the contravariant hermitian form.  Then, the Laplacian $\Delta = d\delta + \delta d$ acts on the $\alpha$ weight space of $\bigwedge^* \fn$ via multiplication by $(\alpha, \alpha + 2\rho)/2$.
\end{thm}
\begin{proof}
By Lemma \ref{lem:Hodge-structure-on-BKM}, the operator $\Delta$, as the graded commutator of $d$ and $\delta$, is a differential operator of order 2 on $\bigwedge^*(\fn)$.  The function taking $\alpha$ to $(\alpha, \alpha + 2\rho)/2$ is a quadratic polynomial on the root lattice, so it suffices to show the assertion of the theorem holds for elements of $\fn$ and $\bigwedge^2 \fn$.

The case of $\fn$ can be done by induction on the number of simple roots into which $\alpha$ decomposes.  If $\alpha$ is simple, then $\fg_\alpha$ lies in the kernel of both $\delta$ and $d$, so $\Delta$ acts by zero, which is equal to $(\alpha, \alpha + 2\rho)/2$.  Now, suppose the claim holds for $\alpha$, and consider the action on $[e_i,x]$ for some $x \in \fg_\alpha$.  By Lemma \ref{lem:d-applied-to-bracket-with-ei}, $d[e_i,x] = \alpha(h_i) e_i \wedge x +  e_i \cdot dx$, so $\Delta [e_i,x] = \alpha(h_i)[e_i,x] + \delta(e_i \cdot dx)$.  Writing $dx = \sum a_j \wedge b_j$, we have
\[ \begin{aligned}
\delta(e_i \cdot dx) &= \sum_j \delta ([e_i,a_j] \wedge b_j + a_j \wedge [e_i,b_j]) \\
&= \sum_j [[e_i,a_j],b_j] + [a_j,[e_i,b_j]] \\
&= \sum_j [e_i, [a_j,b_j]] \\
&= (\ad e_i)\delta(dx) \\
&= \frac{(\alpha, \alpha + 2\rho)}{2} [e_i,x]
\end{aligned} \]
Since $\frac{(\alpha + \alpha_i, \alpha + \alpha_i + 2\rho)}{2} - \frac{(\alpha, \alpha + 2\rho)}{2} = (\alpha,\alpha_i) = \alpha(h_i)$, we conclude that
\[ \Delta [e_i,x] = \frac{(\alpha + \alpha_i, \alpha + \alpha_i + 2\rho)}{2}[e_i,x]. \]
Thus, $\Delta$ acts as expected on the $\alpha + \alpha_i$-weight space.

For $\bigwedge^2 \fn$, we let $x \in \fn_\alpha$ and $y \in \fn_\beta$.  Then, $\Delta(x \wedge y) = \delta d(x\wedge y) + d[x,y] = \delta(dx \wedge y - x \wedge dy) + d[x,y]$.  If $dx = \sum_j a_j \wedge b_j$ and $dy = \sum_k c_k \wedge d_k$, we get
\[ \begin{aligned}
\delta(\sum_j &a_j \wedge b_j \wedge y - \sum_k x \wedge c_k \wedge d_k) \\
&= \sum_j [a_j,b_j] \wedge y - [a_j,y] \wedge b_j + [b_j,y]\wedge a_j \\
&\qquad -(\sum_k [x,c_k]\wedge d_k - [x,d_k]\wedge c_k + [c_k,d_k]\wedge x) \\
&= (\delta dx) \wedge y-(\delta dy)\wedge x + y \cdot dx - x \cdot dy.
\end{aligned} \]
Combining this with Lemma \ref{lem:d-applied-to-any-bracket} yields the answer.
\end{proof}

\begin{cor} \label{cor:Laplacian-on-monster-Lie-algebra} Let $R$ be a subring of $\bC$, and let $\fn$ be the positive subalgebra of the $R$-form of the Monster Lie algebra $\fm_R$ from Definition \ref{defn:R-form-of-monster-Lie-algebra}.  Then, the Laplacian $\Delta = d\delta + \delta d$ acts as $(m-1)n$ on the degree $(m,n)$ part of $\bigwedge^* \fn$.
\end{cor}
\begin{proof}
It suffices to prove this for $R= \bZ$, since the eigenvalues are preserved by base change.  By Corollary \ref{cor:description-of_monster-Lie-algebra}, the contravariant form on $\fm_\bZ$ is positive definite, and the simple roots have the form $(1,n)$, with norm $-2n$, so $(-1,0)$ is a Weyl vector.  Thus, we may apply Theorem \ref{thm:Laplacian}, which asserts that $\Delta$ acts as $\frac{(\alpha + 2\rho,\alpha)}{2} = \frac{((m-2,n),(m,n))}{2} = (m-1)n$ on the degree $(m,n)$ space in $\bigwedge^* \fn$.
\end{proof}

Summing up, we have the following facts about the Monster Lie algebra:

\begin{thm} \label{thm:properties-of-monster-Lie-algebra}
There exists an integral form $\fm_\bZ$ of the Monster Lie algebra, satisfying the following properties:
\begin{enumerate}
\item $\fm_{\bZ}$ is a $\bZ \times \bZ$-graded Borcherds-Kac-Moody Lie algebra over $\bZ$, with an invariant bilinear form that identifies the degree $(0,0)$ subspace with the even unimodular lattice $I\!I_{1,1}$.
\item $\fm_{\bZ}$ has a faithful Monster action by homogeneous inner-product-preserving Lie algebra automorphisms, such that the $\bZ[\bM]$-module structure of the degree $(m,n)$ root space depends only on the value of $mn$, when $\gcd(m,n) = 1$, i.e., if $(m,n)$ is primitive.  In particular, this space is isomorphic to $V^\natural_{1+mn,\bZ}$.
\item The positive subalgebra $\fn_\bZ = \bigoplus_{m \geq 1, n \geq -1} \fm_{m,n,\bZ}$ has $\bZ_{\geq 0} \times \bZ_{\geq 0} \times \bZ$-graded homology $H^*(\fn_\bZ,\bZ) = \bigoplus_{i \geq 0, m \geq 0, n \in \bZ} H^i(\fn_\bZ,\bZ)_{m,n}$.  Furthermore, for any subgroup $G$ of $\bM$, and any $i \geq 0$, and all degrees $(m,n)$ such that $(m-1)n$ is coprime to $|G|$, the representation $H^i(\fn_\bZ,\bZ)_{m,n}$ of $G$ lies in the torsion ideal of the representation ring $\Rep_\bZ(G)$.  In particular, applying any ring homomorphism $\phi: \Rep^\natural_\bZ(G) \to \bC$ to $H^i(\fn_\bZ,\bZ)_{m,n}$ yields zero.
\end{enumerate}
\end{thm}
\begin{proof}
The first two claims are given in Corollary \ref{cor:description-of_monster-Lie-algebra} and Theorem \ref{thm:monster-Lie-is-a-BKM-over-Z}.  For the last claim, Corollary \ref{cor:Laplacian-on-monster-Lie-algebra} asserts the Laplacian acts as $(m-1)n$ on the degree $(m,n)$ part of $\bigwedge^* \fn$, and Corollary
\ref{cor:cancellation-for-order-coprime-to-laplacian} implies the Chevalley-Eilenberg complex has homology in the torsion ideal of $\Rep_\bZ(G)$.
\end{proof}

\section{Quasi-replicability}

\subsection{Adams operations}

We will analyze the exterior powers of $\fn$ using Adams operations, which are linear maps on the representation ring.

\begin{defn} \label{defn:adams-operations}
Let $R$ be a subring of $\bC$, and let $G$ be a finite group.  For any $R$-torsion-free $R[G]$-module $X$ of finite rank, we define
\[ {\bigwedge}_{-q}X = \sum_{n \geq 0} (-q)^n \wedge^n X \in \Rep^R(G)[q]. \]
We define the \textbf{Adams operations} $\{ \psi^n\}_{n \geq 1}$ on $\Rep^R(G)$ by setting $\psi^n(X)$ to be the $q^n\frac{dq}{q}$ coefficient of $-d\log \bigwedge_{-q}X \in \Rep^R(G)[[q]]dq$.  That is, we let 
\[ \sum_{n>0} \psi^n(X) q^n\frac{dq}{q} = -\frac{d(\bigwedge_{-q} X)}{\bigwedge_{-q}X} \] 
On the rationalized representation ring, this can be rewritten as the identity
\[ {\bigwedge}_{-q}X = \exp(-\sum_{n > 0} \psi^n(X)q^n/n). \]
in $(\Rep^R(G) \otimes \bQ)[[q]]$.
\end{defn}

\begin{lem}
The Adams operations satisfy the following properties:
\begin{enumerate}
\item If $X$ is a module of dimension 1, then $\psi^n(X) = X^{\otimes n}$.
\item $\psi^n(X\oplus Y) = \psi^n(X) \oplus \psi^n(Y)$
\item $\psi^n$ commutes with base change of the coefficient ring.
\item $\Tr(g|\psi^n(X)) = \Tr(g^n|X)$.
\item Newton's formula holds: $\sum_{j=0}^{n-1} (-1)^{j+1} (\psi^{n-j} V) (\wedge^j V) = (-1)^n n \wedge^n V$, and we obtain the following recursion: $\psi^n V = \sum_{j=1}^{n-1} (-1)^{j+1} (\psi^{n-j} V) (\wedge^j V) - (-1)^n n \wedge^n V$.
\end{enumerate}
\end{lem}
\begin{proof}
We prove these in the order they are stated:
\begin{enumerate}
\item For this, we just use the definition: The right side is $\frac{Xdq}{1-qX}$, so we get
\[ \sum_{n>0} \psi^n(X) q^n \frac{dq}{q} = \sum_{n>0} X^{\otimes n} q^n \frac{dq}{q}. \] 
\item This follows from the fact that $\bigwedge_{-q}$ takes direct sums to products, together with the Leibniz rule.
\item This follows from the fact that formation of the exterior algebra commutes with base change.
\item We may base change to a ring with enough roots of unity, then split $X$ into one dimensional eigenmodules for $g$.
\item Multiplying both sides of the defining relation by $\bigwedge_{-q}$ and applying the substitution $t = -q$ yields
\[ \left( \sum_{k=1}^\infty (-1)^{k-1}\psi^k(V) t^{k-1}\right)\left( \sum_{j=0}^\infty \wedge^j V t^j \right) = \sum_{m=1}^\infty m \wedge^m V t^{m-1}, \]
so Newton's formula follows from comparing the coefficients attached to $t^{n-1}$.  The recursion comes from isolating the $j=0$ term.
\end{enumerate}
\end{proof}

\begin{rem} \label{rem:adams-not-so-good}
When $R$ is a field of characteristic zero, we also have $\psi^n(X \otimes Y) = \psi^n(X) \otimes \psi^n(Y)$ and $\psi^m(\psi^n(X)) = \psi^{mn}X$, so in particular, the Adams operations form a commutative monoid of ring homomorphisms generated by $\psi^p$ for $p$ prime.  However, these identities fail for $R$ a general commutative ring.  Explicit counterexamples are given in Remark \ref{rem:adams-counterexample}.
\end{rem}

We then have our version of the twisted denominator identity:

\begin{thm} \label{thm:vanishing-in-exponential}
Let $R$ be a subring of $\bC$, let $G$ be a subgroup of $\bM$, and let $\phi: \Rep^\natural_R(G) \to \bC$ be a ring homomorphism.  Then, for any $(a,b)$ for which $(a-1)b$ is a unit in $R' = \bigcap_{P \supset |G|R} R_P$, the $p^a q^b$ term in 
\[ \exp \left( -\sum_{k=1}^\infty \frac{1}{k} \left(\sum_{m=1}^\infty \sum_{n=-1}^\infty \phi(\psi^k(\fm_{m,n,R})) p^{km} q^{kn} \right) \right) \]
vanishes.
\end{thm}
\begin{proof}
Corollary \ref{cor:Laplacian-on-monster-Lie-algebra} asserts that for all $(a,b)$, the Laplacian acts on $(\bigwedge^* \fn)_{a,b}$ by $(a-1)b$, and under our assumption, this quantity is a unit in $R'$.  Thus, by Corollary \ref{cor:cancellation-for-order-coprime-to-laplacian}, $\sum (-1)^i \phi((\bigwedge^i \fn)_{a,b}) = 0$.  Expanding $(\bigwedge^* \fn)_{a,b}$ in terms of Adams operations yields the claim.
\end{proof}

\begin{defn} \label{defn:quasi-replicates}
Let $R$ be a subring of $\bC$, and let $G$ be a subgroup of $\bM$.  Then, for any ring homomorphism $\phi: \Rep^R(G) \to \bC$, we define the functions
\[ T_\phi^{[m]}(\tau) = \sum_{ad = m} \frac{1}{a} \sum_{n \geq -1} \phi(\psi^a(\fm_{d,n,R})) q^{an} \]
for all $m \geq 1$.
\end{defn}


\begin{prop} \label{prop:feature-of-exterior}
$ -\sum_{k=1}^\infty \frac{1}{k} \left(\sum_{m=1}^\infty \sum_{n=-1}^\infty \phi(\psi^k(\fm_{m,n,R})) p^{km} q^{kn} \right) = -\sum_{m>0}p^m T_\phi^{[m]}(\tau)$ 
\end{prop}
\begin{proof}
\[ \begin{aligned}
-\sum_{k=1}^\infty \frac{1}{k} \left(\sum_{m=1}^\infty \sum_{n=-1}^\infty \phi(\psi^k(\fm_{m,n,R})) p^{km} q^{kn} \right)
&= -\sum_{m=1}^\infty \sum_{a|m} \frac{1}{a} \sum_{n=-1}^\infty \phi(\psi^a(\fm_{m/a,n,R})) p^m q^{an} \\
&= -\sum_{m=1}^\infty p^m \sum_{ad = m} \frac{1}{a} \sum_{n \geq -1} \phi(\psi^a(\fm_{d,n,R})) q^{an} \\
&= -\sum_{m>0}p^m T_\phi^{[m]}(\tau)
\end{aligned} \]
\end{proof}

\begin{cor} \label{cor:vanishing-coefficients}
The $p^a q^b$ coefficient of
\[ \exp \left( -\sum_{m>0}p^m T_\phi^{[m]}(\tau) \right) \]
vanishes for all $(a-1)b$ coprime to $|G|$.
\end{cor}
\begin{proof}
By Proposition \ref{prop:feature-of-exterior}, we have  
\[ \exp \left( -\sum_{m>0}p^m T_\phi^{[m]}(\tau) \right) = \exp \left( -\sum_{k=1}^\infty \frac{1}{k} \left(\sum_{m=1}^\infty \sum_{n=-1}^\infty \phi(\psi^k(\fm_{m,n,R})) p^{km} q^{kn} \right) \right) \]
and Theorem \ref{thm:vanishing-in-exponential} asserts that the $p^a q^b$ term vanishes when $(a-1)b$ is coprime to $|G|$.
\end{proof}

\subsection{Quasi-replicability, main result}

\begin{defn} \label{defn:quasi-replicable}
A periodic holomorphic function on the complex upper half-plane $f(\tau) = q^{-1} + \sum_{n \geq 0} a_n q^n$ is \textbf{quasi-replicable} of exponent $N$ if for $m \geq 1$, there exist holomorphic functions $f^{[m]}(\tau)$ on  the complex upper half-plane, satisfying the following conditions:
\begin{enumerate}
\item $f^{[1]}(\tau) = f(\tau)$.
\item $f^{[m]}(\tau)$ is periodic in $\tau$ with $q$-expansion $f^{[m]}(\tau) = \frac{1}{m}q^{-m} + \sum_{n > -m} a_n^{[m]} q^n$.
\item If $\gcd(m,n) = 1$, then the $q^n$ coefficient of $f^{[m]}(\tau)$ is equal to the $q^{mn}$ coefficient of $f(\tau)$, i.e., $a_n^{[m]} = a_{mn}^{[1]}$.
\item the $p^a q^b$ coefficient of
\[ \exp \left( -\sum_{m>0}p^m f^{[m]}(\tau) \right)
\]
vanishes for all $(a-1)b$ coprime to $N$. 
\end{enumerate}
\end{defn}

We now have our main theorem:

\begin{thm} \label{thm:main}
Let $R$ be a subring of $\bC$, and let $G$ be a subgroup of $\bM$.  Then, for any ring homomorphism $\phi: \Rep_R^\natural(G) \to \bC$, the ``generalized McKay-Thompson series'' $T_\phi(\tau) = \sum_{n \geq 0} \phi(V^\natural_{n,R}) q^{n-1}$ is quasi-replicable of exponent $|G|$.
\end{thm}
\begin{proof}
We set $f^{[m]}(\tau) = T_\phi^{[m]}(\tau) = \sum_{ad = m} \frac{1}{a} \sum_{n \geq -1} \phi(\psi^a(\fm_{d,n,R})) q^{an}$ following Definition \ref{defn:quasi-replicates}.  Then, the first condition follows from the isomorphism $\fm_{1,n,R} \cong V^\natural_{1+n,R}$ given in Corollary \ref{cor:description-of_monster-Lie-algebra}.  The second and third conditions follow immediately from the defining formula for $T_\phi^{[m]}(\tau)$ together with the fact that the unique real simple root of $\fm$ has degree $(1,-1)$ (i.e., $\fm_{m,n,R} = 0$ when $mn < -1$).  The fourth condition is precisely Corollary \ref{cor:vanishing-coefficients}.
\end{proof}

\begin{rem}
The quasi-replicability condition is related to the notion of ``replicability'' introduced in \cite{CN79} and further explained in \cite{N84}, in the following way: If $f(\tau)$ is replicable, then there are uniquely defined holomorphic functions $f^{(k)}(\tau)$ for all $k \geq 1$, such that the following equality holds in some neighborhood of the cusp $(i \infty, i \infty)$ in the product of two complex upper half-planes:
\[ f(\sigma) - f(\tau) = p^{-1} \exp\left(-\sum_{m =1}^\infty \frac{p^m}{m}\sum_{ad=m, 0 \leq b < d} f^{(a)}\left(\frac{a\tau+b}{d} \right) \right), \]
where $p = e^{2\pi i \sigma}$.  Setting $f^{[m]}(\tau) = \frac{1}{m}\sum_{ad=m} f^{(a)}(\frac{a\tau+b}{d})$ we find that replicable functions are quasi-replicable of exponent 1.  There is also a notion of ``finite order'' replicability, defined by the condition that the ``replicates'' $f^{(k)}$ are periodic in $k$.  We can define ``quasi-replicate'' functions $T_\phi^{(k)}(\tau) = \sum_{n \geq -1} \phi(\psi^k(V^\natural_n))q^{n-1}$, and these are periodic in $k$ when the Adams operations $\psi^k$ are periodic.  However, it appears that we do not have enough control over $\fm_{m,n}$ for non-primitive vectors $(m,n)$ to turn this into a good, precise analogue of finite order replicability.
\end{rem}

\section{Explicit results}

\subsection{Cyclic subgroups of order 4 in the Monster}

Let $G$ be a cyclic group of order 4.  The indecomposable $\bZ[G]$-modules were classified into 9 isomorphism types by Roiter \cite{R60}, and independently in the Ph. D. dissertations \cite{K62} and \cite{T61}.  The tensor structure is given in \cite{R65}, and we use this to classify the homomorphisms $\phi: \Rep_\bZ(G) \to \bC$.

\begin{defn}
We fix notation for the 9 isomorphism classes of indecomposable $\bZ[G]$-modules.  Let $g$ be a generator of $G$.
\begin{enumerate}
\item $A$ is the trivial module $\bZ$.
\item $B$ is $\bZ$ with $g$ acting as $-1$.
\item $C$ is the rank 2 module $\bZ x \oplus \bZ y$ with $gx = y, gy = -x$.
\item $D$ is the group ring $\bZ[G]$
\item $E$ is the rank 2 module $\bZ x \oplus \bZ y$ with $gx = y, gy = x$.
\item $C^A$ is the unique non-split extension of $C$ by $A$.
\item $C^B$ is the unique non-split extension of $C$ by $B$.
\item $C^E$ is the non-split extension of $C$ by $E$ that is not the group ring.
\item $C^{AB}$ is the non-split extension of $C$ by $A \oplus B$.
\end{enumerate}
\end{defn}

\begin{thm} \label{thm:order-4-ring-homs}
Let $\phi: \Rep_\bZ(G) \to \bC$ be a ring homomorphism.  Then $\phi$ takes the indecomposable modules $(A,B,C,D,E,C^A,C^B, C^E, C^{AB})$ to one of the following tuples:
\begin{enumerate}
\item $(1,1,2,4,2,3,3,4,4)$, from the rank, or ``trace of 1'' homomorphism.
\item $(1,1,-2,0,2,-1,-1,0,0)$, from the ``trace of $g^2$'' homomorphism.
\item $(1, 1, 2, 0, 2, 1, 1, 2, 2)$, from the ``total dimension of Tate cohomology of $g^2$'' homomorphism.
\item $(1, -1, 0, 0, 0, 1, -1, 0, 0)$, from the ``trace of $g$'' homomorphism.
\item $(1, -1, 0, 0, 0, -1, 1, 0, 0)$, $(1, 1, 0, 0, 0, 1, 1, 0, 0)$, and $(1, 1, 0, 0, 0, -1, -1, 0, 0)$.  We call these ``twisted versions'' of the trace of $g$, because they are the same up to sign.
\item $(1,1, 0, 0, 0, 1, 1, 2, 2)$, from an exotic function.
\end{enumerate}
In particular, $\phi$ takes values in integers.
\end{thm}
\begin{proof}
We reproduce the multiplication table from \cite{R65}:

\vspace{2mm}

\begin{tabular}{c|c|c|c|c|c|c|c|c|c}
& $A$ & $B$ & $C$ & $D$ & $E$ & $C^A$ & $C^B$ & $C^E$ & $C^{AB}$ \\ \hline
$A$ & $A$ & $B$ & $C$ & $D$ & $E$ & $C^A$ & $C^B$ & $C^E$ & $C^{AB}$ \\
$B$ &  & $A$ & $C$ & $D$ & $E$ & $C^B$ & $C^A$ & $C^E$ & $C^{AB}$ \\
$C$ & & & $2E$ & $2D$ & $2C$ & $D+E$ & $D+E$ & $C+D+E$ & $C+D+E$ \\
$D$ & & & & $4D$ & $2D$ & $3D$ & $3D$ & $4D$ & $4D$ \\
$E$ & & & & & $2E$ & $C+D$ & $C+D$ & $C+D+E$ & $C+D+E$ \\ \hline
$C^A$ & & & & & & $A+2D$ & $B+2D$ & $C^{AB} + 2D$ & $C^E + 2D$ \\
$C^B$ & & & & & & & $A + 2D$ & $C^{AB} + 2D$ & $C^E + 2D$ \\
$C^E$ & & & & & & & & $C^E + C^{AB} + 2D$ & $C^E + C^{AB} + 2D$ \\
$C^{AB}$ & & & & & & & & &  $C^E + C^{AB} + 2D$ 
\end{tabular}

\vspace{2mm}

\noindent\textbf{Case 1}:
If $\phi(D)$ is nonzero, then by $D \otimes X \cong D^{\oplus \rank X}$, we get the rank homomorphism.

\noindent\textbf{Case 2}:
We assume $\phi(D) = 0$.  Since the tensor squares of $C$ and $E$ are both $2E$, we consider the case $\phi(E) = 2$.  Then, $\phi(C)$ can be $2$ or $-2$.

\noindent\textbf{Case 2a}:
Suppose $\phi(D) = 0$ and $\phi(C) = 2$.  Then, the remaining values of $\phi$ are uniquely determined by tensoring with $C$ and $E$, and we get the ``total dimension of Tate cohomology of $g^2$'' homomorphism.

\noindent\textbf{Case 2b}:
Suppose $\phi(D) = 0$ and $\phi(C) = -2$.  Then, the remaining values of $\phi$ are uniquely determined by tensoring with $C$ and $E$, and we get the ``trace of $g^2$'' homomorphism.

\noindent\textbf{Case 3}:
Suppose $\phi(D) = \phi(E) = \phi(C) = 0$.  Then, $\phi(B)^2 = \phi(C^A)^2 = \phi(C^B)^2 = 1$, and $\phi(C^B) = \phi(B)\phi(C^A)$, so we split into 4 cases depending on the signs of $\phi(B)$ and $\phi(C^A)$.  Because $\phi(C^{AB})\phi(C^A) = \phi(C^{AB})\phi(C^B) = \phi(C^E)$, we find that $\phi(B) = -1$ implies $\phi(C^{AB}) = \phi(C^E) = 0$.  If $\phi(B) = 1$ and $\phi(C^A) = -1$, then $\phi(C^E)^2 = \phi(C^E + C^{AB} + 2D) = 0$, so once again $\phi(C^{AB}) = \phi(C^E) = 0$.  These three cases yield the ``trace of $g$'' homomorphism and  two twisted versions.

\noindent\textbf{Case 3'}: If $\phi(B) = \phi(C^A) = 1$, then $\phi(C^E) = \phi(C^{AB})\phi(C^A) = \phi(C^{AB})$, so $\phi(C^E)^2 = \phi(C^E + C^{AB} + 2D) = 2\phi(C^E)$.  We conclude that $\phi(C^E)$ is $0$ or $2$.  These two values yield the remaining maps.
\end{proof}

\begin{rem}
The indecomposable modules $C^E$ and $C^{AB}$ cannot be distinguished by homomorphisms from the representation ring.  However, we will see that they never appear in the decomposition of $V^\natural_\bZ$ under any order 4 automorphism.
\end{rem}

We now consider the decomposition of $V^\natural_\bZ$ under the action of automorphisms of order 4.

\begin{lem} \label{lem:restriction-from-order-4-to-order-2}
The indecomposable $\bZ[G]$-modules restrict to $\langle g^2 \rangle$ in the following way, where $I$ denotes the rank 1 module with $g^2$ acting by $-1$, and $H$ denotes the subgroup $\langle g^2 \rangle$:

\vspace{2mm}

\begin{tabular}{c|ccccccccc}
& $A$ & $B$ & $C$ & $D$ & $E$ & $C^A$ & $C^B$ & $C^E$ & $C^{AB}$ \\ \hline
$H$-rep. & $\bZ$ & $\bZ$ & $2I$ & $2\bZ[H]$ & $2\bZ$ & $\bZ[H] + I$ & $\bZ[H] + I$ & $\bZ + I + \bZ[H]$ & $\bZ + I + \bZ[H]$
\end{tabular}

\vspace{2mm}

Furthermore, if $G$ is generated by an element $g$ of order 4 in $\bM$, then we have the following lists of possible indecomposable representations in $V^\natural_{n,\bZ}$:
\begin{enumerate}
\item $g$ in class 4A: $A$, $B$, $D$, $E$ for $n$ even, and $C$, $D$, $C^A$, $C^B$ for $n$ odd.
\item $g$ in class 4B: $A$, $B$, $D$, $E$ for all $n$.
\item $g$ in class 4C: $A$, $B$, $D$, $E$ for $n$ even, and $C$, $D$, $C^A$, $C^B$ for $n$ odd.
\item $g$ in class 4D: $A$, $B$, $D$, $E$ for $n$ even, and $C$, $D$, $C^A$, $C^B$ for $n$ odd.
\end{enumerate}
In particular, the indecomposable representations $C^E$ and $C^{AB}$ do not appear.
\end{lem}
\begin{proof}
Using cases 1 and 2 in the proof of Theorem \ref{thm:order-4-ring-homs}, we obtain the decomposition.

By Theorems 5.2, 5.3 in \cite{BR96}, if $g^2$ lies in class 2A (i.e., for class 4B), then $V^\natural_{n,\bZ}$ decomposes into $\bZ$ and $\bZ[H]$ for all $n$, and if $g^2$ lies in class 2B (i.e., for classes 4A, 4C, 4D), then $V^\natural_{n,\bZ}$ decomposes into $\bZ$ and $\bZ[H]$ for $n$ even and $I$ and $\bZ[H]$ for $n$ odd.  The lack of terms combining $\bZ$ and $I$ eliminates $C^E$ and $C^{AB}$ from possibility.
\end{proof}

We continue with class 4A, where we have the most complete information.  The key property that helps us is that $T_{2B} = q^{-1} + 276q - 2048q^2 + 11202q^3 - \cdots$ while $T_{4A} = q^{-1} + 276q + 2048q^2 + 11202q^3 + \cdots$, i.e., the coefficients of $T_{2B}$ have the same size as those of $T_{4A}$, but alternate sign.

\begin{thm} \label{thm:4A}
Let $G$ be a cyclic subgroup of $\bM$ generated by an element in class 4A.  Then, $V^\natural_\bZ$ decomposes as a direct sum of the indecomposable $\bZ[G]$-modules $A$, $D$, and $C^A$, and they generate a subring of $\Rep_\bZ(G)$ isomorphic to $\bZ[d,c]/(d^2 - 4d, c^2-2d-1,(c-3)d)$.  There are exactly 3 homomorphisms from this ring to $\bC$, corresponding to the traces of elements in classes 1A, 2B, and 4A.  The multiplicities $a_n, d_n, c_n$ of the indecomposable modules $A,D,C^A$ in $V^\natural_{n,\bZ}$ are given by the generating function formula
\[ \left(\begin{array}{c} \sum a_n q^{n-1} \\ \sum d_n q^{n-1} \\ \sum c_n q^{n-1} \end{array} \right) = \left(\begin{array}{ccc} 0 & 1/2 & 1/2 \\ 1/4 & 1/4 & -1/2 \\ 0 & -1/2 & 1/2 \end{array} \right) \left(\begin{array}{c}T_{1A}(\tau) \\ T_{2B}(\tau) \\ T_{4A}(\tau) \end{array} \right) . \]
\end{thm}
\begin{proof}
The trace of $g$ is equal to the trace of $g^2$ on $V^\natural_n$ when $n$ is even, so $B$ and $E$ cannot appear.  The trace of $g$ is minus the trace of $g^2$ on $V^\natural_n$ when $n$ is odd, so $C$ and $C^B$ cannot appear.  This leaves the possibilities we listed, and the structure of the ring generated by these indecomposable modules is given by the tensor products listed in the proof of Theorem \ref{thm:order-4-ring-homs}.
\end{proof}

For the remaining classes, we do not have comprehensive multiplicity information, because there is at least one $T_\phi$ that we don't know how to evaluate.  However, we can put a bound on the ``order'' of quasi-replicability, by using periodicity of Adams operations.

\begin{prop}
Let $G$ be a cyclic subgroup of $\bM$ generated by an element of order 4.  Then, the Adams operations on the indecomposable $\bZ[G]$-modules appearing in $V^\natural$ are periodic with period at most 8.
\end{prop}
\begin{proof}
We first compute the exterior powers of indecomposable modules: they can be distinguished for all cases except $\Lambda^2 D$, by considering eigenvalues of the corresponding complex representations and the decomposition under the action of $H = \langle g^2 \rangle$, and for the remaining case by similarity for matrices mod 2 (we checked this with a SAGE computation \cite{SAGE}).  A brief calculation with log derivatives yields the values and periodicity we want.

\vspace{2mm}

\begin{tabular}{c|ccccccc}
 & $A$ & $B$ & $C$ & $D$ & $E$ & $C^A$ & $C^B$  \\ \hline
$H$-rep. & $\bZ$ & $\bZ$ & $2I$ & $2\bZ[H]$ & $2\bZ$ & $\bZ[H] + I$ & $\bZ[H] + I$ \\
$\Lambda^2$ & $0$ & $0$ & $A$ & $C+D$ & $B$ & $C^A$ & $C^A$ \\
$\Lambda^3$ & $0$ & $0$ & $0$ & $D$ & $0$ & $A$ & $B$ \\
$\Lambda^4$ & $0$ & $0$ & $0$ & $B$ & $0$ & $0$ & $0$ \\
$\psi^{2k+1}$ & $A$ & $B$ & $C$ & $D$ & $E$ & $C^A$ & $C^B$ \\
$\psi^{4k+2}$ & $A$ & $A$ & $2E-2A$ & $2D-2C$ & $2E-2B$ & $\scriptstyle{A + 2D - 2C^A}$ & $\scriptstyle{A + 2D - 2C^B}$ \\
$\psi^{8k+4}$ & $A$ & $A$ & $2A$ & $4E-4B$ & $2A$ & $3A$ & $3A$ \\
$\psi^{8k}$ & $A$ & $A$ & $2A$ & $4A$ & $2A$ & $3A$ & $3A$ \\
\end{tabular}
\end{proof}

\begin{rem} \label{rem:adams-counterexample}
As we mentioned in Remark \ref{rem:adams-not-so-good}, the Adams operations are neither closed under composition, nor do they give us endomorphisms of representation rings.  From the table, we see that for composition, $\psi^2(\psi^2(C)) = \psi^2(2E-2A) = 4E-4B-2A$ while $\psi^4(C) = 2A$, and for multiplication, we have $\psi^2(C^2) = \psi^2(2E) = 4E-4B$ while $(\psi^2(C))^2 = (2E-2A)^2 = 4A$
\end{rem}

\begin{cor}
Let $G$ be a cyclic subgroup of $\bM$ generated by an element of order 4, and let $R$ be a subring of $\bC$.  Then, for any ring homomorphism $\Rep^\natural_R(G) \to \bC$, the series $T_\phi$ is quasi-replicable of exponent 2, and the ``quasi-replicate'' functions $T_\phi^{(k)}(\tau) = \sum_{n \geq -1} \phi(\psi^k(V^\natural_n))q^{n-1}$ are periodic in $k$ with period $8$.
\end{cor}
\begin{proof}
This follows immediately from the 8-periodicity of Adams operations and the definition of $T_\phi^{(k)}$.
\end{proof}

Our best remaining case is 4B, with at most 1 extra function:

\begin{prop}
Let $G$ be a cyclic subgroup of $\bM$ generated by an element in class 4B.  Then, $V^\natural_\bZ$ decomposes as a direct sum of the indecomposable $\bZ[G]$-modules $A$, $D$, $E$, and possibly $B$.  $\Rep^\natural_\bZ(G)$ is either (if $B$ is not present) $\bZ[d,e]/((d-4)d, (e-2)d, (e-2)e)$ or (if $B$ is present) $\bZ[b,d,e]/(b^2 - a, (b-1)d, (d-4)d, (b-1)e, (e-2)d, (e-2)e)$.  If $B$ is not present, then there are 3 homomorphisms from this ring to $\bC$, corresponding to the traces of elements in classes 1A, 2A, and 4B.  If $B$ is present, then there are 4 homomorphisms to $\bC$, given by the 3 traces together with a twisted version of trace taking $(A,B,D,E)$ to $(1,1,0,0)$.  For this last map $\phi$, the coefficients of $T_\phi$ are bounded above and below by the corresponding coefficients of $T_{2A}(\tau)$ and $T_{4B}(\tau)$.  Specifically, if we denote the multiplicities of the $\bZ[G]$-modules $A,B,D,E$ in $V^\natural_{n,\bZ}$ by $a_n, b_n, d_n, e_n$, we have the following relations between these numbers and the $q^{n-1}$ coefficients of generalized McKay-Thompson series:
\begin{enumerate}
\item $T_{1A}(\tau) = \sum_n (a_n + b_n + 4d_n + 2e_n)q^{n-1} = q^{-1} + 0 + 196884q + 21493760q^2 + 864299970q^3 + \cdots$
\item $T_{2A}(\tau) = \sum_n (a_n + b_n + 2e_n)q^{n-1} = q^{-1} + 0 + 4372q + 96256q^2 + 1240002q^3 + \cdots$
\item $T_{4B}(\tau) = \sum_n (a_n - b_n)q^{n-1} = q^{-1} + 0 + 52q + 0q^2 + 834q^3 + \cdots$
\item $T_\phi(\tau) = \sum_n (a_n + b_n)q^{n-1} = q^{-1} + 0 + (a_2+b_2)q + (a_3 + b_3)q^2 + \cdots$
\end{enumerate}
In particular, the multiplicity of $D$ in $V^\natural_{n,\bZ}$ is $d_n$, and this is given by the $q^{n-1}$ coefficient of $\frac{T_{1A}(\tau) - T_{2A}(\tau)}{4} = 48128q + 5349376q^2 + 215764992q^3 + \cdots$.
\end{prop}
\begin{proof}
The indecomposable modules are given in Lemma \ref{lem:restriction-from-order-4-to-order-2}, and the tensor products and homomorphisms are listed in the proof of Theorem \ref{thm:order-4-ring-homs}.
\end{proof}

We note that the last claim in this proposition implies $\bZ[G]$ makes up a large proportion of the indecomposable summands in $V^\natural_n$, and this proportion approaches 1 as $n \to \infty$.

Let us consider further what we can know about this extra function $T_\phi$.  Quasi-replicability gives us an identity for each $p^{2i}q^{2j+1}$, where $0 < i < j$.  For example, vanishing of $(\bigwedge^* \fn_{\bZ_2})_{1,2k+1}$ (equivalently, vanishing of the $p^2 q^{2k+1}$ term in the product) yields $V^\natural_{4k+3,\bZ_2} = V^\natural_{2k+3,\bZ_2} \oplus \bigoplus_{i=1}^k V^\natural_{i+1,\bZ_2} \otimes V^\natural_{2k+2-i,\bZ_2}$, and applying $\phi$ yields $a_{4k+3} + b_{4k+3} = a_{2k+3} + b_{2k+3} + \sum_{i=1}^k (a_{i+1} + b_{i+1})(a_{2k+2-i} + b_{2k+2-i})$.  We my use this to establish some asymptotic behavior.

\begin{prop}
For even $n$, we have $c_1 n^{-3/4} e^{\pi \sqrt{2n}} < a_n < c_2 n^{-3/4} e^{\pi \sqrt{8n}}$ for some strictly positive constants $c_1,c_2$.  Furthermore, either $a_n = 0$ for all odd $n$, or for any $\epsilon > 0$, there is some $N$ such that for all $n \equiv 3 \pmod 4$ satisfying $n>N$, we have $a_n > e^{\pi\sqrt{2n}(1-\epsilon)}$.
\end{prop}
\begin{proof}
Theorem 8.11 of \cite{DGO15} gives an exact formula for the coefficients of McKay-Thompson series, and we find that the $q^{2k+1}$-st coefficient of $T_{4B}$ grows like $2^{-1/4}(2k+1)^{-3/4} e^{\pi \sqrt{2(2k+1)}}$, and the corresponding coefficient of $T_{2A}$ grows like $2^{-3/4}(2k+1)^{-3/4}e^{\pi\sqrt{8(2k+1)}}$.  This latter figure gives us our upper bound for all $k$, and our lower bound for $a_{2k+2}$.  The recursion formula for $\phi$ implies $a_{4k+3} \geq a_{2k+3} + \sum_{i=1}^k a_{i+1} a_{2k+2-i}$, so $a_{4k+3}$ is strictly positive for all $k \geq \frac{n-3}{2}$.  Now, let $\epsilon' > 0$ satisfy $a_{i+1} > e^{\pi\sqrt{2i}(1-\epsilon')}$ for the unique $i \equiv 2 \pmod{4}$ satisfying $i \in \{k-1, k, k+1, k+2\}$.  Then, $2k+2-i$ is even, so $a_{2k+2-i} > 2^{-1/2} k^{-3/4} e^{\pi\sqrt{2k}}$.  Applying the recursion formula, we find that $a_{4k+3} > 2^{-1/2} k^{-3/4}e^{\pi(\sqrt{2k}+ \sqrt{2k-2} - \epsilon\sqrt{2k+4})}$.  For $k$ sufficiently large, this is greater than $e^{\pi\sqrt{2(4k+3)}(1-\epsilon'')}$ for $\epsilon'' < \frac{2 \epsilon'}{3}$.  Iterating the substitution $k \mapsto 4k+3$, $\epsilon' \mapsto \epsilon''$, we eventually find $\epsilon'' < \epsilon$.
\end{proof}

The classes 4C and 4D are similar but have more potential classes and functions.  The analysis for elements of orders 9 and 25 is also similar, but there are substantially more indecomposable modules to consider, and the computation of Adams operations seems to require a lot of computer memory.

\subsection{Cyclic subgroups of order 6}

We begin by reviewing the results of \cite{K62}, which gives an algorithm for classifying indecomposable representations of cyclic groups of square-free order.  Let $p,q$ be distinct primes.  The irreducible representations of $\bZ/pq\bZ$ over $\bZ$ are represented by matrices of the form $\Gamma_1, \Gamma_p, \Gamma_q, \Gamma_{pq}$, where each $\Gamma_r$ has size $\phi(r)$ and equivalence classes are parametrized by ideal classes in $\bZ[\zeta_r] = \bZ[e^{2\pi i/r}]$.  Indecomposable representations of $\bZ/pq\bZ$ over $\bZ$ are represented by block upper triangular matrices, whose diagonal blocks have at most one copy of each type $\Gamma_r$, where $r$ is weakly increasing as we progress down the diagonal.  Two indecomposable representations of this form are equivalent if and only if one can be taken to the other by a sequence of ``sigma'' transformations (conjugation by strictly block upper triangular matrices) and ``delta'' transformations (conjugation by block diagonal matrices).

For any representing block-upper triangular matrix $W$, we form a decorated graph whose vertices are labeled with the values of $r$ attached to the diagonal blocks, where $r < r'$ are connected by an edge if and only if there is a nonzero entry in the block that shares rows with $\Gamma_r$ and columns with $\Gamma_{r'}$.  
A representation is indecomposable if and only if all matrices in the equivalence class have connected graphs.  Within each equivalence class of indecomposable representations, there is a representative with minimal graph, and in this graph $r$ is adjacent to $r'$ only if $r'/r$ is a prime.

\begin{lem} \label{lem:graphs-determine-reps}
If 2 indecomposable $\bZ[\bZ/6\bZ]$-modules yield identical minimal graphs, then they are equivalent.
\end{lem}
\begin{proof}
Because we are concerned with the case $p=2, q=3$, the relevant ideal class groups are trivial.  A brief computation in SAGE shows that for each decorated connected graph with vertex labels taken from a subset of $\{1,2,3,6\}$, there is a unique orbit under the action of sigma and delta transformations, such that this graph is the minimal graph of the equivalence class \cite{SAGE}.
\end{proof}

We write $\Gamma(r_1,r_2,\ldots)$ to denote an indecomposable representation whose block diagonals are $\Gamma_{r_1}, \Gamma_{r_2},\ldots$.  We note that in general, this notation may not determine the representation up to equivalence.

\begin{lem} \label{lem:classification-of-6A-indecomposables}
Let $\rho: \bZ/6\bZ \to \Aut V^\natural_\bZ$ be a representation taking a generator to an element in class 6A.  Then, any indecomposable submodule has one of the following forms: $\Gamma_1$, $\Gamma(1,2)$, $\Gamma(1,3)$, $\Gamma(1,2,3)$, and $\Gamma(1,2,3,6)$.  Furthermore, the only submodule of the form $\Gamma(1,2,3,6)$ is isomorphic to the group ring $\bZ[\bZ/6\bZ]$, and its decorated graph is a 4-cycle $1-2-6-3-1$.
\end{lem}
\begin{proof}
The nontrivial powers of 6A elements lie in classes 2A and 3A, so Modular Moonshine for these classes \cite{BR96}, \cite{C17} implies the restriction to subgroups of orders 2 and 3 have no indecomposable constituents equivalent to the augmentation ideals.  Restriction to a subgroup of order 2 erases the $1-3$ and $2-6$ edges, and we eliminate those graphs that have isolated $2$ or $6$.  This eliminates $\Gamma_2, \Gamma_6, \Gamma(2,6), \Gamma(1,2,6), \Gamma(2,3,6)$.  Restriction to a subgroup of order 3 erases the $1-2$ and $3-6$ edges, and we eliminate those graphs that have isolated $3$ or $6$.  This eliminates $\Gamma_3, \Gamma_6, \Gamma(3,6), \Gamma(1,3,6), \Gamma(2,3,6)$.  We are left with the graph types we claimed.  In particular, $\Gamma(1,2,3,6)$ must be given by a cycle, and it is straightforward to see that the group ring gives suitable connectivity.
\end{proof}

\begin{lem} \label{lem:tensor-products-for-6A}
We have the following tensor products of indecomposable representations appearing in Lemma \ref{lem:classification-of-6A-indecomposables} (abbreviating $\Gamma(1,2,3,6)$ as $D$): 

\vspace{2mm}

\begin{tabular}{c|c|c|c|c|c}
& $\Gamma_1$ & $\Gamma(1,2)$ & $\Gamma(1,3)$ & $\Gamma(1,2,3)$ & $D$ \\ \hline
$\Gamma_1$ & $\Gamma_1$ & $\Gamma(1,2)$ & $\Gamma(1,3)$ & $\Gamma(1,2,3)$ & $D$ \\
$\Gamma(1,2)$ & & $2\Gamma(1,2)$ & $D$ & $D + \Gamma(1,2)$ & $2D$ \\
$\Gamma(1,3)$ & & & $3\Gamma(1,3)$ & $D + 2\Gamma(1,3)$ & $3D$ \\
$\Gamma(1,2,3)$ & & & & $2D + \Gamma(1,3) + \Gamma_1$ & $4D$ \\
$D$ & & & & & $6D$ 
\end{tabular}
\end{lem}
\begin{proof}
This follows from restriction to the subgroups of orders 2 and 3, together with the tensor product computations in \cite{B98}.
\end{proof}

\begin{prop} \label{prop:classification-of-6A-homomorphisms}
Let $\phi: \Rep^\natural_\bZ(\bZ/6\bZ) \to \bC$ be a ring homomorphism induced by sending a generator $g$ to class 6A.  Then $\phi$ takes the tuple $(\Gamma_1,\Gamma(1,2),\Gamma(1,3),\Gamma(1,2,3),D)$ of indecomposable modules to one of the following tuples:
\begin{enumerate}
\item $(1,2,3,4,6)$, from the rank, or ``trace of 1'' homomorphism.
\item $(1,0,3,2,0)$, from the ``trace of $g^3$'' homomorphism.
\item $(1,2,0,1,0)$, from the ``trace of $g^2$'' homomorphism.
\item $(1, 0, 0, -1, 0)$, from the ``trace of $g$'' homomorphism.
\item $(1, 0, 0, 1, 0)$, from a twisted version of the trace of $g$.
\end{enumerate}
In particular, $\phi$ takes values in integers.
\end{prop}
\begin{proof}
From Lemma \ref{lem:tensor-products-for-6A}, $D \otimes X = (\dim X)D$ so $\phi(D)$ is $0$ or $6$.  If it is $6$, we get the rank homomorphism, so from now on we assume $\phi(D) = 0$.  Let $h_2 = \phi(\Gamma(1,2))$ and $h_3 = \phi(\Gamma(1,3))$.  Since $\Gamma(1,2)^2 = 2\Gamma(1,2)$, $\Gamma(1,3)^2 = 3\Gamma(1,3)$ and $\Gamma(1,2)\otimes \Gamma(1,3) = 0$, we have $(h_2, h_3)= (2,0), (0,3), (0,0)$.

Case $(2,0)$:  $\Gamma(1,2,3)$ is sent to $1$ because of tensoring with $\Gamma(1,2)$, so we get $\Tr(g^2)$.

Case $(0,3)$: By essentially the same argument as the $(2,0)$ case, we get $\Tr(g^3)$.

Case $(0,0)$: The tensor square of $\Gamma(1,2,3)$ is taken $1$, so $\Gamma(1,2,3)$ is taken to $\pm 1$.  We get $\Tr(g)$ from $-1$, and a twisted version of trace from $1$.
\end{proof}

\begin{thm} \label{thm:6A}
Let $\phi: \Rep^\natural_\bZ(\bZ/6\bZ) \to \bC$ be a ring homomorphism attached to a cyclic group generated by a 6A element.  Then, the power series $T_\phi(\tau)$ is one of $T_{1A}, T_{2A}, T_{3A}, T_{6A}$ or some quasi-replicable function of exponent 6, whose $q^n$ coefficient is a non-negative integer bounded below by the $q^n$ coefficient of $T_{6A} = q^{-1} + 79q + 352q^2 + 1431q^3 + \cdots$ and bounded above by the $q^n$ coefficient of $T_{3A} = q^{-1} + 783q + 8672q^2 + 65367q^3 + \cdots$.
\end{thm}
\begin{proof}
By Proposition \ref{prop:classification-of-6A-homomorphisms}, we have 5 possible ring homomorphisms, and 4 of them are traces, yielding the McKay-Thompson series attached to powers of $g$.  The last homomorphism  yields a power series whose coefficients satisfy the claimed bounds, because the evaluation on any indecomposable module satisfies the same bounds.  Quasi-replicability follows from Theorem \ref{thm:main}.
\end{proof}

\begin{conj}
The extra function is equal to $T_{6A}$.  Equivalently, $\Gamma(1,2,3)$ does not appear in the 6A-decomposition of $V^\natural_\bZ$.
\end{conj}

As we mentioned in the introduction, we expect similar behavior for all elements of type $pq$A for $p,q$ distinct primes - these are the elements whose powers are all Fricke-invariant.  The main difficulty is proving a suitable substitute for Lemma \ref{lem:graphs-determine-reps}.  For other elements of order $pq$, or more generally other elements of square-free composite order, we have less control over the coefficients of the functions $T_\phi$ and indecomposable submodules, but we can still put bounds on multiplicities.

\section{Open problems}

\begin{enumerate}
\item Are quasi-replicable functions ``almost all modular''?  If we remove the ``quasi'' prefix, we know that replicable functions of finite order are either ``modular fictions'' (i.e., of the form $q^{-1} + aq$ for $a=0$ or $a$ a root of unity) or Hauptmoduln of finite level, when their coefficients are algebraic integers \cite{C08}.  In fact, if $f(\tau) = \sum_n a_n q^n$ is replicable of finite order, and $\sum_{n \in \bZ} n|\sigma(a_n)|^2 > 1$ for some automorphism $\sigma$ of $\bC$, then $f$ is a Hauptmodul.
\item For which $R$ and $G$ are the Adams operations for the $R[G]$-module structure on $V^\natural_R$ periodic?  That is, for which group rings do we have some $N$ such that $\psi^k(V^\natural_{n,R}) = \psi^{k+N}(V^\natural_{n,R})$ for all $n,k \geq 0$?  This is a natural question in light of the previous one, since it implies the finite order property on the quasi-replicable functions.  When $|G|$ is invertible in $R$, periodicity follows from the fact that all homomorphisms are traces of elements.
\item Does every ring homomorphism $\Rep^\natural_R(G) \to \bC$ take each $V^\natural_{n,R}$ to an algebraic integer?  This may help when trying to show that a series is a Hauptmodul.
\item How bad can the ring $\Rep^\natural_R(G)$ get?  We could optimistically hope that it is finite rank for $R=\bZ$ and $G = \bM$, and we could pessimistically consider the possibility that even for the minimal examples $R \cong \bZ_2$ and $G = \bZ/2\bZ \times \bZ/2\bZ$ or $G = \bZ/8\bZ$ where $\Rep_R(G)$ is infinitely generated, $\Rep^\natural_R(G)$ may also be infinitely generated \cite{N61}, \cite{HR62}, \cite{HR63}.  The case of a cyclic group generated by a 4A element, where by Theorem \ref{thm:4A} we see only 3 out of a possible 9 indecomposable representations of $\bZ/4\bZ$, is evidence (perhaps weak) that some significant simplification can happen.  Specific question: what is the smallest subgroup $G$ of $\bM$, if one exists, such that $\Rep^\natural_\bZ(G)$ is infinitely generated?
\item For which subgroups $G < \bM$ and subrings $R \subset \bC$ is $V^\natural_R$ asymptotically regular as an $R[G]$-module?  That is, do the maximal free $R[G]$ summands $M_n \subset V^\natural_{n,R}$ satisfy $\lim_{n \to \infty} \frac{\dim M_n}{\dim V^\natural_n} = 1$?  By Corollary 8.2 in \cite{DGO15}, this is true for $R = \bC$ and $G$ arbitrary, and this is shown in more generality in \cite{AB19}.  The results of Modular Moonshine also imply asymptotic regularity for arbitrary $R$ and $G$ cyclic of prime order.  Our results in section 6 imply asymptotic regularity for $R$ any subring of $\bC$ and $G$ a cyclic group generated by an element of type 4A, 4B, or 6A.  
\item What if instead of considering homomorphisms to $\bC$ to make coefficients of power series, we allowed for a target ring that didn't annihilate torsion, such as $\Rep^\natural_R(G)$ itself?  This would let us consider finer-grained information about representations, and in particular we could distinguish isomorphism types in ideal classes.  However, if we consider modular forms with torsion coefficients, then the notion of Hauptmodul in this setting may need revision.
\item Can one construct a theory of twisted modules attached to general homomorphisms from representation rings?  This seems like the most promising potential route to a general modularity result, generalizing \cite{DLM97}.
\end{enumerate}

\noindent{\textbf{Acknowledgements}} This work was supported by JSPS Kakenhi Grant-in-Aid for Young Scientists (B) 17K14152.

\end{document}